\documentclass[pdflatex,sn-mathphys-num]{sn-jnl}
\usepackage[english]{babel} 
\usepackage[utf8]{inputenc} 

\usepackage{graphicx}%
\usepackage{multirow}%
\usepackage{amsmath,amssymb,amsfonts}%
\usepackage{amsthm}%
\usepackage{mathrsfs}%
\usepackage[title]{appendix}%
\usepackage{xcolor}%
\usepackage{textcomp}%
\usepackage{manyfoot}%
\usepackage{booktabs}%
\usepackage{algorithm}%
\usepackage{algorithmicx}%
\usepackage{algpseudocode}%
\usepackage{listings}%

\usepackage{psfrag}
\usepackage{verbatim}
\usepackage{multirow}
\usepackage{bm}
\usepackage{float}
\usepackage{subfigure}
\usepackage{extarrows}
\usepackage{booktabs}
\usepackage{appendix}
\usepackage[symbol]{footmisc}
\usepackage{enumerate}
\usepackage{latexsym}
\usepackage{color}
\usepackage{epstopdf}
\usepackage{diagbox}
\usepackage[symbol]{footmisc}

\newtheorem{lemma}{Lemma}[section]
\newtheorem{theorem}{Theorem}[section]
\newtheorem{Assumption}{Assumption}[section]

\theoremstyle{remark}
\newtheorem{remark}{Remark}[section]
\numberwithin{equation}{section}

\newcommand{\ba}{\begin{array}}\newcommand{\ea}{\end{array}}
\newcommand{\be}{\begin{eqnarray}}
	\newcommand{\ee}{\end{eqnarray}}
\newcommand{\beq}{\begin{equation}}\newcommand{\eeq}{\end{equation}}
\newcommand{\bex}{\begin{eqnarray*}}
	\newcommand{\eex}{\end{eqnarray*}}
\def\L{\mathcal L}
\def\Dt{\Delta t}

\begin{document}

\title[A randomized Milstein--Galerkin Finite Element method for SEEs]{A Drift-Randomized Milstein--Galerkin Finite Element Method for Semilinear Stochastic  Evolution Equations}


\author[1]{\fnm{Xiao} \sur{Qi}}\email{qixiao@jhun.edu.cn}

\author[2]{\fnm{Yue} \sur{Wu}}\email{yue.wu@strath.ac.uk}


\author*[3]{\fnm{Yubin} \sur{Yan}}\email{y.yan@chester.ac.uk}


\affil[1]{\orgdiv{School of Artificial Intelligence}, \orgname{Jianghan University}, \orgaddress{\street{No. 8 Sanjiaohu Road, Economic Development Zone}, \city{Wuhan}, \postcode{430056}, \state{Hubei}, \country{China}}}

\affil[2]{\orgdiv{Department of Mathematics and Statistics}, \orgname{University of Strathclyde}, 
		\city{Glasgow}, \postcode{G1 1XH}, 
		\country{UK}}
	

\affil*[3]{\orgdiv{School of Computer and Engineering Sciences}, \orgname{University of Chester}, \orgaddress{\street{Parkgate Road}, \city{Chester}, \postcode{CH1 4BJ}, 
		\country{UK}}}




\abstract{\quad Kruse and Wu [{\it Math. Comp.} {\bf 88} (2019) 2793--2825] proposed a fully discrete randomized Galerkin finite element method for semilinear stochastic evolution equations (SEEs) driven by additive noise and showed that this method attains a temporal strong convergence rate exceeding order $\frac{1}{2}$ without imposing any differentiability assumptions on the drift nonlinearity. They further discussed a potential extension of the randomized method to SEEs with multiplicative noise and introduced the so-called \textit{drift-randomized Milstein--Galerkin finite element} fully discrete scheme, but without providing a corresponding strong convergence analysis. 
	
\quad This paper aims to fill this gap by  rigorously analyzing the strong convergence behavior of the {\it drift-randomized Milstein-Galerkin finite element scheme}. By avoiding the use of differentiability assumptions on the nonlinear drift term, we establish strong convergence rates in both space and time for the proposed method. The obtained temporal convergence rate is $\mathcal{O}(\Dt^{1-\varepsilon_0})$, where $\Dt$ denotes the time step size and $\varepsilon_0$ is an arbitrarily small positive number. Numerical experiments are reported to validate the theoretical findings.}




\maketitle

\section{Introduction}
\label{Intro}
\numberwithin{equation}{section}
Let $T\in(0,\infty)$, and let $D\subset\mathbb{R}^d$, $d\in\{1,2,3\}$, be a bounded, convex polygonal or polyhedral  spatial domain with  boundary $\partial D$. Define $H:=L^2(D)$ as the separable Hilbert space equipped with inner product $(\cdot,\cdot)_H$ and norm $\|\cdot\|_{H}$. $W(t)$ Denote by $(\Omega_{W},\mathcal{F}^W,\{\mathcal{F}^W_{t}\}_{0\leq t\leq T},\mathbb{P}_W)$ a filtered probability space satisfying the usual conditions. Let $W(t)$ be an $\mathcal{F}^W_{t}$-adapted $H$-valued Wiener process with covariance operator $Q$, where $Q$ is a symmetric, positive-definite, trace-class operator, with  orthonormal eigenfunctions $\{\phi_j\in H: j\in\mathbb{N_+}\}$ and corresponding positive eigenvalues $\{q_j, j\in\mathbb{N_+}\}$. Precisely, $W(t)$ admits the Karhunen-Lo\`eve expansion:
\be\label{K-Lexpansion}
W(t)=\sum_{j=1}^{\infty}\sqrt{q_j}\phi_j\beta_j(t),
\ee
with $\beta_j(t)$ representing the independent and identically distributed standard Brownian motions.

Denote by $A: \operatorname{dom}(A) \subset H\to H$ a self-adjoint, positive definite, second order linear elliptic operator with compact inverse and zero Dirichlet boundary condition. For $s\in\mathbb{R}$, define the Hilbert space $\dot{H}^s:=\operatorname{dom}(A^{\frac{s}{2}})$ as the domain of the fractional power of the operator $A$, equipped with the norm 
\bex \|v\|_{\dot{H}^s}:=\|A^{\frac{s}{2}}v\|_{H},\ \ \forall v\in\dot{H}^s.
\eex
Let $L^{p}(\Omega_W,\dot{H}^s)$, $p\ge2$, denote the space of $\dot{H}^s$-valued random variables that are $p$-fold integrable, endowed with the norm
\be\label{LpOmega}
\| v\|_{{L^{p}(\Omega_W, \dot{H}^s)}}:=\big(\mathbb{E}_W[\| v \|_{{\dot{H}^s}}^{p}]\big)^{\frac{1}{p}}<\infty,\ \ \forall v\in L^{p}(\Omega_W,\dot{H}^s).   
\ee
Here, $\mathbb{E}_W[\cdot]$ denotes the expectation with respect to the probability measure $\mathbb{P}_W$.

Let's consider the semilinear stochastic evolution equations (SEEs) driven by multiplicative noise:
\begin{equation}\label{underlyingproblem}
	\begin{aligned}
		\mathrm{d}u(t)&=\big(-Au+F(u(t))\big)\mathrm{d}t+G(u(t))\mathrm{d}W(t),\    t\in(0,T],\\
		 u(0)&=u_0,
	\end{aligned}
\end{equation}
where the initial value $u_0:\Omega_{W}\to H$ satisfies $u_0\in L^p(\Omega_{W}, \dot{H}^{1+r})$ for $p\ge 2$ and $r\in[0,1]$. The drift term $F(u)$ is assumed to satisfy the global Lipschitz and linear growth conditions, specifically, there exists a constant $c>0$ such that
\begin{equation}\label{Fcondition}
	\begin{aligned}
		&\|F(v_1)-F(v_2)\|_{H}\leq c\|v_1-v_2\|_{H},\ \ \forall v_1, v_2\in H,\\
		&\|F(v)\|_{H}\leq c(1+\|v\|_{H}),\ \ \forall v\in H.
	\end{aligned}
\end{equation}
Throughout this paper, the symbol $c$, with or without subscripts, denotes a generic positive constant, whose value may not be the same at different occurrences. The nonlinearty $G(\cdot)$ represents the diffusion term, and additional assumptions on  $G(\cdot)$  are specified in Section \ref{sec2}. 

Our objective is to approximate the mild solution of the stochastic evolution problem \eqref{underlyingproblem}. According to the variation-of-constants formula, the mild solution (cf. \cite[Definition 10.22]{LORD2014}) admits the representation
\begin{equation}\label{for17}
	u(t)=S(t)u_0+\int_{0}^{t}S(t-\tau)F(u(\tau))d\tau+\int_{0}^{t} S(t-\tau)G(u(\tau)) dW(\tau),
\end{equation}
where $S(t)=e^{-tA}$ denotes the analytic semigroup generated by the infinitesimal generator $-A$. 

Several standard numerical methods for approximating the solution of \eqref{underlyingproblem} are discussed in the literature \cite{jentzen2011taylor, kruse2014strong, LORD2014, zhang2017numerical,cui2022density,cui2025stochastic,chen2023accelerated} and the references therein. Due to the presence of stochastic noise, classical Maruyama-type methods typically yield low strong temporal convergence rates, often limited to order $\frac{1}{2}$. A widely used approach to improve accuracy involves constructing higher-order schemes based on It{\^o}-Taylor expansions  \cite{jentzen2011taylor}. Under sufficient smoothness of the coefficient functions, such expansions theoretically enable the derivation of schemes with arbitrarily high temporal order. However, these high-order methods may exhibit numerical instability and generally require the evaluation of higher-order derivatives and simulation of iterated stochastic integrals, which substantially restrict their practicality. More critically, the regularity assumptions required for such constructions are often too stringent in infinite-dimensional settings. For example, \cite[Chapter 8]{jentzen2011taylor} assumes that the drift function $F$ belong to $C_b^{\infty}(H,H)$, i.e., it is infinitely Fr{\'e}chet differentiable with bounded derivatives---an assumption rarely satisfied for genuinely nonlinear Nemytskii-type operators. To remove this limitation, several works, including \cite{jentzen2015milstein, kruse2014consistency, leonhard2018enhancing, wang2015exponential, wang2017strong}, propose relaxing the smoothness requirements to  $F\in C^2_b(\dot{H}^s, H)$ or $F\in C^2_b(H,\dot{H}^{-s})$ for $s\in(0,1]$. Although this relaxation broadens the applicability of higher-order methods such as Milstein and Wagner-Platen schemes \cite{becker2016exponential, barth2013lp, barth2012milstein, liu2021strong,wang2024strong2,liu2023strong}, these methods still fundamentally rely on the Fr{\'e}chet differentiability of the drift term.

In recent years, a high-order drift-randomized Milstein time discretization method, initially proposed in \cite{kruse2019randomized1}, has attracted significant attention due to its ability to establish strong convergence results without relying on Taylor expansions or derivative information of the nonlinear drift term. A number of extensions and variants of this drift-randomization strategy have been developed for stochastic ordinary differential equations (SODEs); see, e.g., \cite{bao2025randomised, przy2024approximation, biswas2024explicit, przy2024randomized, morkisz2021randomized, vashistha2025first, liao2022truncated} and the references therein. In contrast, the application of drift-randomized methods to stochastic partial differential equations (SPDEs) remains relatively scarce. Kruse and Wu \cite{kruse2019randomized} introduced a fully discrete scheme that combines a standard Galerkin finite element method with a randomized Runge-Kutta time discrete approach for SPDEs with additive noise. Remarkably, the authors demonstrated that the temporal strong convergence rate of the proposed scheme matches that of the classical semi-implicit Milstein--Galerkin method \cite{kruse2014consistency}, while avoiding any differentiability assumptions on the drift nonlinearity. Moreover, in \cite[Remark 5.11]{kruse2019randomized}, the authors discussed a possible extension of the randomized approach to SEEs with multiplicative noise, and introduced the corresponding scheme, referred to as the \textit{drift-randomized Milstein--Galerkin finite element method}. However, no strong convergence analysis was provided for this scheme in the multiplicative noise setting. To the best of our knowledge, such an analysis has not yet available in the existing literature. The current work aims to fill this gap by providing a detailed strong convergence analysis of this \textit{drift-randomized Milstein--Galerkin finite element method}.

For convenience, we henceforth abbreviate the \textit{drift-randomized Milstein--Galerkin finite element} method as DRMGFE. Now we describe the DRMGFE scheme in detail. Let $\Delta t := T/N$ denote a uniform time step for a positive integer  $N$, and define $t_n:=n\Dt$ for $n=1,\dots,N$. Let $\mathcal{T}_h$, with $h\in(0,1)$, denote a shape-regular triangulation of the spatial domain $D$. Define finite element space $V_h$ by
\bex
V_h:=\lbrace v \in C^0(\bar{D}),\ v=0 \ \mbox{on} \ \partial D, \  v|_{K}\in {\color{blue} \mathscr{P}_1(K)}\ \mbox{for all} \ K\in\mathcal{T}_h \rbrace \subset \dot{H}^1,
\eex
where {\color{blue}$\mathscr{P}_1(K)$} denotes the set of polynomials of degree at most one on each element $K$. Let $\mathcal{P}_h: H \to V_h$ be the $L^2$-orthogonal projection onto $V_h$, and define $A_h: V_h \to V_h$ as the discrete version of the unbounded operator $A$. For example, if $A := -\Delta$ is the negative Laplacian subject to homogeneous Dirichlet boundary conditions, then $A_h$ is defined by $(A_hw,v):=(\nabla w, \nabla v),\  \forall w,v\in V_h$. 

Let sequence $(\xi_n)_{n\in\{1,\dots,N\}}$ denote the independent family of  $\mathcal{U}(0, 1)$-distributed random variables defined on another probability space $(\Omega_{\xi},\mathcal{F}^{\xi}, \{\mathcal{F}^{\xi}_t\}_{0\leq t\leq T},\mathbb{P_\xi})$. Note that $(\xi_n)_{n\in\{1,\dots,N\}}$ induces a natural filtration $\{\mathcal{F}^{\xi}_n\}_{n\in\{0,1,\dots,N\}}$ on the probability space $(\Omega_{\xi},\mathcal{F}^{\xi},\mathbb{P_\xi})$ by setting $\mathcal{F}^{\xi}_0:=\{\emptyset,\Omega_{\xi}\}$ and $\mathcal{F}^{\xi}_n:=\sigma\{\xi_j: j=1,\dots, n\}$. Here $\sigma\{\xi_j: j=1,\dots, n\}$ denotes the $\sigma$-algebra generated by collection of random variables  $(\xi_j)_{j\in\{1,\dots, n\}}$. Define the product probability space
\bex
(\Omega,\mathcal{F},\mathbb{P}):=(\Omega_{W}\times\Omega_{\xi},\mathcal{F}^W\times\mathcal{F}^{\xi},\mathbb{P}_W\times\mathbb{P_\xi}),
\eex
with the corresponding expectations denoted by $\mathbb{E}_W$ and $\mathbb{E}_\xi$, respectively. Let 
\bex
\pi_{\Dt}:=\{t_n=n\Dt: n=0,1,\dots,N\}
\eex
denote the set of temporal grid points. We define the discrete-time filtration  $\{\mathcal{F}^{\pi_{\Dt}}_n\}_{n\in\{0,\dots,N\}}$ on $(\Omega,\mathcal{F},\mathbb{P})$ by setting  
\bex
\mathcal{F}^{\pi_{\Dt}}_n:=\mathcal{F}^W_{t_n}\times\mathcal{F}^{\xi}_n, \ \ \mbox{for}\  n\in\{0,1,\dots,N\}.
\eex

The proposed DRMGFE scheme is given via the following recursion: for $n=1,\dots,N$, 
\begin{equation}\label{randomizedMilstein}
	\begin{aligned}
		u_{h,\xi_n}^n&=S_{h,\xi_n\Dt}\big(u_h^{n-1}+\xi_n\Dt F(u_h^{n-1})+G(u_h^{n-1})\Delta W^{n-1}_{\xi_n}\big),\\
		u_h^n&=S_{h,\Dt}\big(u_h^{n-1}+\Dt F(u_{h,\xi_n}^n)+\int_{t_{n-1}}^{t_n}\mathscr{G}(\tau,u_h^{n-1})dW(\tau)\big),\\
		u_h^0&=\mathcal{P}_h u_0,
	\end{aligned}
\end{equation}
where $u_h^{n}$ denotes the fully discrete approximation to $u(t_n)$, with $t_n=n\Dt$. The operators 
\begin{equation}\label{Shdt}
S_{h,\Dt}:=(I+\Dt A_h)^{-1}\mathcal{P}_h \ \ \mbox{and} \ \  S_{h,\xi_n\Dt}:=(I+\xi_n\Dt A_h)^{-1}\mathcal{P}_h
\end{equation}
represent the resolvents at deterministic and randomized time steps, respectively. The intermediate value  $u_{h,\xi_n}^n$ serves as 
an approximation to the solution at the random time point  $t_{n,\xi_n}:=t_{n-1}+\xi_n\Dt$, where each $t_{n,\xi_n}\sim\mathcal{U}(t_{n-1}, t_n)$ independently. Moreover, the increment $\Delta W^{n-1}_{\xi_n}$ and the process $\mathscr{G}(\tau,u_h^{n-1})$ are given by
\begin{equation*}
	\begin{aligned}
\Delta W^{n-1}_{\xi_n}&:=W(t_{n-1}+\xi_n\Dt)-W(t_{n-1}),\\ \mathscr{G}(\tau,u_h^{n-1})&:=G(u_h^{n-1})+G'(u_h^{n-1})G(u_h^{n-1})(W(\tau)-W(t_{n-1})),
	\end{aligned}
\end{equation*}
for $t_{n-1}\leq\tau\leq t_{n}$. In addition, it is worth noting that the time discretization used in \cite{kruse2019randomized} for additive noise-driven SEEs corresponds to the Maruyama-type scheme. In the additive noise case, the Maruyama-type scheme can be viewed as a special instance of the Milstein scheme, and therefore its temporal strong convergence order coincides with that of the Milstein method (see, e.g., \cite[Chapter 10]{LORD2014}). By contrast, in the multiplicative noise setting, the strong convergence analysis of the scheme \eqref{randomizedMilstein} becomes substantially more involved. In particular, additional assumptions on the diffusion coefficient $G(\cdot)$ and its derivatives are required, and the presence of the stochastic correction term appeared in $\mathscr{G}(\cdot)$ further complicates the error analysis when compared with the additive noise case.

Indeed, the time discretization in \eqref{randomizedMilstein} can be regarded as a two-stage Runge--Kutta scheme, where the second stage introduces a randomized node in the drift term $F(\cdot)$. Similar node-randomized Runge--Kutta methods have been successfully applied to the numerical solution of complex deterministic differential equations; see, e.g., \cite{hofmanova2020randomized, eisenmann2019randomized} and references therein. However, there has been limited work on combining the randomized Runge--Kutta approach with finite element discretization to construct fully discrete scheme for SEEs and to rigorously analyze its strong convergence property; see, e.g., \cite{kruse2019randomized}. To the best of our knowledge, the present work seems the first attempt to investigate the strong convergence behavior of the fully discrete scheme \eqref{randomizedMilstein}. Notably, when $\xi_n = 0$ for all $n \in \{1, \dots, N\}$, the scheme \eqref{randomizedMilstein} reduces to the classical semi-implicit Milstein--Galerkin finite element scheme, as studied in \cite{kruse2014consistency,kruse2014strong} and references therein. In contrast to classical convergence analyses, which typically rely on Taylor expansions and require differentiability assumptions on the nonlinear drift term $F(\cdot)$, the introduction of the randomized internal stage $u_{h,\xi_n}^n$ allows us to establish high-order temporal strong convergence rates without imposing such assumptions for $F$. This is achieved via the use of martingale inequalities in the error analysis of the drift integral (see, e.g., \cite[Lemma 5.7]{kruse2019randomized}).

The main contribution of this work lies in establishing strong convergence rates for the fully discrete scheme \eqref{randomizedMilstein} in both time and space. In particular, without imposing any differentiability assumptions on the nonlinear drift term, we prove that the numerical solution $u_h^n$ converges strongly to the mild solution $u(t_n)$ with rate $\mathcal{O}(\Delta t^{1 - \varepsilon_0} + h^{2 - \varepsilon_0})$ for all $n = 1, \dots, N$, where $\Delta t$ and $h$ denote the temporal and spatial mesh sizes, respectively, and $\varepsilon_0 > 0$ is arbitrarily small. This result aligns with the classical estimate $\mathcal{O}(\Delta t^{\frac{1+r}{2}} + h^{1+r})$ for the semi-implicit Milstein--Galerkin finite element method when $r \to 1$, as reported in \cite[Theorem 1.1]{kruse2014consistency}.

The remainder of the paper is organized as follows. In Section \ref{sec2}, we introduce the basic notations and preliminaries used throughout the paper. Section \ref{convergenceanalysis} presents a detailed analysis of the strong convergence rates for the fully discrete scheme \eqref{randomizedMilstein}. Finally, Section \ref{numericalresult} reports numerical experiments that confirm the theoretical results and demonstrate the efficiency of the proposed scheme.

\section{Preliminaries}\label{sec2}
We start by introducing the notations that will be used throughout the paper.  Let $\mathcal{L}(H)$ denote the space of bounded linear operators from $H$ to itself. The Cameron-Martin space associated with the covariance operator $Q$ is given by $Q^{\frac{1}{2}}(H) := \{ Q^{\frac{1}{2}}v : v \in H \}$. Define the space ${\L_0^2}:=\{L: Q^{\frac{1}{2}}(H)\to H\}$ as the collection of Hilbert-Schmidt operators from $Q^{\frac{1}{2}}(H)$ to $H$ \cite{yan2005galerkin}, equipped with norm
\be\label{HilbertHimitnorm}
\| L \|_{{\L_0^2}}:=\Big(\sum_{j=1}^\infty \big\| L Q^\frac{1}{2}\phi_j \big\|_{{H}}^2\Big)^\frac{1}{2}<\infty,\ \ \forall L\in{\L_0^2}.
\ee
Additionally, let $HS(Q^{\frac{1}{2}}(H),\L_0^2)$ denote another Hilbert-Schmidt operator space equipped with the norm
\be\label{anotherHilbert}
\|L\|_{HS(Q^{\frac{1}{2}}(H),\L_0^2)}:=\Big(\sum_{i=1}^{\infty}\sum_{j=1}^{\infty}q_iq_j\big\|L\phi_i\phi_j\big\|^2_{H}\Big)^{\frac{1}{2}}<\infty,\ \ \forall L\in HS(Q^{\frac{1}{2}}(H),\L_0^2).
\ee
Let $\L(H,\L^2_0)$ denote the space of bounded linear operators from $H$ to $\L_0^2$, endowed with the norm  
\bex\|L\|_{\L(H,\L_0^2)}:=\sup\limits_{v\in H}\frac{\|L v\|_{{\L_0^2}}}{\| v\|_{{H}}}<\infty, \ \ \forall L\in\L(H,\L_0^2).
\eex 
Moreover, define the tensor product operator space $\L^{\otimes2}(H,\L_0^2):=\L(H\otimes H,\L_0^2)$ with the norm
\bex
\| L \|_{{\L^{\otimes2}(H,\L_0^2)}}:=\sup\limits_{v_1,v_2\in H}\frac{\|L v_1v_2\|_{{\L_0^2}}}{\| v_1\|_{{H}}\| v_2\|_{{H}}}<\infty,\ \ \forall L\in\L^{\otimes2}(H,\L_0^2).
\eex
To establish the well-posedness of the problem \eqref{underlyingproblem} and to facilitate the subsequent error analysis of the proposed numerical scheme, we impose the following assumptions on the nonlinear diffusion term $G(\cdot)$.
\begin{Assumption}\label{Gassumption}
	Suppose that the nonlinear mapping $G(\cdot): H\to\L_0^2$ satisfies a commutativity-type condition as described in \cite[Assumption 3, Remark 1]{jentzen2015milstein} \footnote[2]{This commutativity condition is naturally satisfied by a broad class of SPDEs; see, e.g., \cite[Section 4]{jentzen2015milstein}.},  and 
		\begin{equation}\label{lipofG}
					\big\|G(v_1)-G(v_2)\big\|_{{\L_0^2}}
			\leq c\big\|v_1-v_2\big\|_{H},\ \ \forall v_1,v_2\in H.
		\end{equation}
		 Furthermore, for each $s\in[0,\frac{1}{2}]$, suppose that the mapping	$A^sG(\cdot):\operatorname{dom}(A^s)\to {\L_0^2} $ satisfies the growth condition
	\begin{align}\label{linearofG}
		\| A^{s}G(v)\|_{{{\L_0^2}}}\leq c\big(1+\big\|A^s v\big\|_{H}\big),\ \ \forall v\in \operatorname{dom}(A^s).
	\end{align}
	Additionally, assume that the first and second Fr{\'e}chet derivatives of $G$ satisfy:  $G'(v)\in\L(H,\L_0^2)$ and $G''(v)\in\L^{\otimes2}(H,\L_0^2)$ for $v\in H$, and that there exists a constant $c>0$ such that
	\be\label{deriveg1}
	\|G'(v)\|_{{\L(H,\L_0^2)}}+\|G''(v)\|_{{\L^{\otimes2}(H,\L_0^2)}}\leq c,\ \ \forall v\in H.
	\ee
	Moreover, assume that $G'(v)G(v)\in HS(Q^{\frac{1}{2}}(H),\L_0^2)$ for $v\in H$, and that
	\begin{equation}\label{deriveg2}
		\|G'(v_1)G(v_1)-G'(v_2)G(v_2)\|_{HS(Q^{\frac{1}{2}}(H),\L_0^2)} \leq c\|v_1-v_2\|_{{H}},\ \ \forall v_1,v_2\in H.
	\end{equation}
\end{Assumption}
\begin{remark}
	The conditions \eqref{lipofG} and \eqref{linearofG} are standard in the numerical analysis of SEEs with multiplicative noise; see, e.g., \cite[Assumption 2.2]{kruse2014optimal}. The derivative-based assumptions \eqref{deriveg1} and \eqref{deriveg2} were employed in \cite[Assumption 5.2]{liu2021strong} in the analysis of Milstein-type schemes, where specific examples of nonlinearities $G(\cdot)$ satisfying these conditions can be found in \cite[Example 5.1]{liu2021strong}.
\end{remark}
Under assumptions \eqref{Fcondition}, \eqref{lipofG}, and \eqref{linearofG}, the existence and uniqueness of a mild solution $u: [0,T] \times \Omega_W \to H$ to \eqref{underlyingproblem} can be established, see, e.g., \cite[Theorem 1]{jentzen2012regularity}. Furthermore, by the regularity results in \cite[Theorem 2.27, Theorem 2.31]{kruse2014strong}, it holds true that for $p\ge 2$, $\gamma\in[0,1)$,
	\begin{align}
		&\sup_{0\leq t\leq T}\mathbb{E}_W\big[\|u(t)\|_{\dot{H}^{1+\gamma}}^p\big]<\infty,\label{uregularity}\\
		&\|u(t_1)-u(t_2)\|_{L^{p}(\Omega_{W},\dot{H}^\gamma)}\leq c(t_1-t_2)^{\frac{1}{2}}, \ \ 0\leq t_1\leq t_2\leq T. \label{holderregularity}
	\end{align}
We also recall several standard smoothing estimates for the semigroup operators $S(t)$ and its discrete counterpart $S_{h,\Dt}$, which are used in the subsequent analysis. For $m\in\{1,\dots,N\}$, define $S_{h,\Dt}^m:=(I+\Dt A_h)^{-m}\mathcal{P}_h$. The following bounds hold:
\begin{align}
	\bullet \ \	&	\sup\limits_{\Dt\in(0,T)}\sup\limits_{h\in(0,1)}\|S_{h,\Dt}\|_{\L(H)}\leq 1,\ \   \sup\limits_{\Dt\in(0,T)}\sup\limits_{h\in(0,1)}\sup\limits_{m\in\{1,\dots,N\}}\|S_{h,\Dt}^m\|_{\L(H)}\leq 1.  \label{operatoruniformbound}  \\
	\bullet \ \	&\| A_h^{\rho}S^j_{h,\Dt}v\|_{H}=\|A_h^{\rho}(I+\Dt A_h)^{-j}\mathcal{P}_hv\|_H\leq ct_j^{-\rho}\| v\|_H ,\ \ \forall v\in H,\ \  j\in\{1,\dots,N\},\ \ \rho\ge0.\label{nond-operatoresti}\\
	\bullet \ \	&\mbox{For}\ 0\leq\mu\leq\rho\leq2,\  v\in \dot{H}^{\mu}, \ t\in[t_{m-1},t_m),\  m\in\{1,\dots,N\}, \ \mbox{it} \ \mbox{holds}\ \mbox{that}\notag\\
	&\qquad\qquad\qquad \big\|(S(t)-S_{h,\Dt}^m)v\big\|_{H}\leq c(h^{\rho}+\Dt^{\frac{\rho}{2}})t^{-\frac{\rho-\mu}{2}}\big\|A^{\frac{\mu}{2}}v\big\|_H.\label{d-operatoresti-positive}
\end{align}
We refer to \cite[Lemma 4.1]{kruse2019randomized} for the uniform operator bound in \eqref{operatoruniformbound}, to \cite[Lemma 7.3]{thomee2007galerkin} for the  operator estimate \eqref{nond-operatoresti}, and to \cite[(4.7)]{qi2019optimal} for estimates \eqref{d-operatoresti-positive}.

It is worth emphasizing that scheme~\eqref{randomizedMilstein} is linear at each time step. Indeed, once $u_h^{n-1}$ is given, both $u_{h,\xi_n}^n$ and $u_h^n$ can be computed by applying the resolvent operators $S_{h,\xi_n\Dt}$ and $S_{h,\Dt}$, respectively. The well-posedness of scheme~\eqref{randomizedMilstein} is established in Appendix Lemma~\ref{wellposedness}.

\section{Strong convergence}\label{convergenceanalysis}
Our objective in this section is to derive the error estimate for \bex
\max\limits_{n\in\{0,\dots,N\}}\|u(t_n)-u_h^n\|_{L^2(\Omega,H)},
\eex
where $u(t_n)$ denotes the mild solution of problem \eqref{underlyingproblem} at the time point $t_n=n\Dt$, and $u_h^n$ is the numerical solution obtained from scheme \eqref{randomizedMilstein}, and the $L^2(\Omega, H)$-norm is given by 
 \bex
\|\cdot\|_{L^2(\Omega, H)}:=\big(\mathbb{E_\xi}[\|\cdot\|^2_{L^2(\Omega_{W},H)}]\big)^{\frac{1}{2}}=\big(\mathbb{E_\xi}\big(\mathbb{E}_W[\|\cdot\|_H^2]\big)\big)^{\frac{1}{2}}.
\eex
For the convenience of analysis, we introduce the notations
	\begin{align}
	B_{h,\Delta t}(u_h^{n-1},\xi_n)
	&:= S_{h,\Delta t}\Big(\Delta t F(u_{h,\xi_n}^n)
	+ \int_{t_{n-1}}^{t_n}\mathscr{G}(\tau,u_h^{n-1})\, dW(\tau)\Big),\ \ n=1,\dots,N,\label{Bhdtdiscrete}\\
		B_{h, \Delta t} ( u(t_{n-1}), \xi_{n})
	&:= S_{h, \Delta t} 
	\Big( \Delta t F( u (t_{n-1}, \xi_{n})) + \int_{t_{n-1}}^{t_{n}} \mathscr{G} (\tau, u(t_{n-1})) d W(\tau)\Big),\ \ n=1,\dots,N,\label{Bhdtcontinue}\\
	u(t_{n-1}, \xi_{n})
	& := S_{h, \xi_{n} \Delta t} 
	\Big( u(t_{n-1}) + \xi_{n} \Delta t F( u(t_{n-1})) + G( u(t_{n-1})) \Delta W_{\xi_{n}}^{n-1}\Big),\ \ n=1,\dots,N. \label{PSI}
		\end{align}
Next, we establish an a priori estimate for the difference between the terms $B_{h,\Delta t}(u_h^{j-1},\xi_j)$ and $B_{h,\Delta t}(u(t_{j-1}),\xi_j)$ for $j=1,\dots,n$, which is stated in the following lemma.
\begin{lemma}\label{lemma1}
	For $\Dt\in(0,T)$, $h\in(0,1)$ and $n\in\{1,\dots,N\}$, the following bounds hold 
	\begin{equation}\label{Nonlinearstability}
			\Big\|\sum_{j=1}^n S_{h,\Delta t}^{n-j}
			\big(B_{h,\Delta t}(u(t_{j-1}),\xi_j)
			- B_{h,\Delta t}(u_h^{j-1},\xi_j)\big)\Big\|_{L^2(\Omega,H)}^2
			\le c\Delta t \sum_{j=1}^n (t_n-t_{j-1})^{-1/2}
			\big\|u(t_{j-1})-u_h^{j-1}\big\|_{L^2(\Omega,H)}^2.
	\end{equation}
\end{lemma}
\begin{proof}
By \eqref{Bhdtdiscrete}, \eqref{Bhdtcontinue}, and the expression for $\mathscr{G}(\tau,\cdot)$, we deduce the following bound 
	\begin{equation}
		\begin{aligned}\label{condition32}
			&\Big\|\sum_{j=1}^{n}S_{h,\Dt}^{n-j}\big(B_{h,\Dt}(u(t_{j-1}),\xi_j)-B_{h,\Dt}(u_h^{j-1},\xi_j)\big)\Big\|_{L^2(\Omega,H)}\\
			&\leq \Dt\Big\|\sum_{j=1}^{n}S_{h,\Dt}^{n-j+1}\big(F(u(t_{j-1},\xi_j))-F(u^{j}_{h,\xi_j})\big)\Big\|_{L^2(\Omega,H)}\\
			&\quad+\Dt\Big\|\sum_{j=1}^{n}S_{h,\Dt}^{n-j+1}\big(G(u(t_{j-1}))-G(u_h^{j-1})\big)\big(W(t_j)-W(t_{j-1})\big)\Big\|_{L^2(\Omega,H)}\\
			&\quad+\Dt\Big\|\sum_{j=1}^{n}S_{h,\Dt}^{n-j+1}\int_{t_{j-1}}^{t_j}\Big(\int_{t_{j-1}}^{\tau}
			\big(G'(u(t_{j-1}))G(u(t_{j-1}))-G'(u_h^{j-1})G(u_h^{j-1})\big)dW(s)\Big)dW(\tau)\big)\Big\|_{L^2(\Omega,H)}\\&=:J_1+J_2+J_3.
		\end{aligned}
	\end{equation}
	Next we estimate $J_1$, $J_2$, $J_3$, respectively. For $J_1$, we use \eqref{Fcondition}, \eqref{nond-operatoresti} with $\rho=\frac{1}{2}$, \eqref{PSI}, and the inequality $\|A_h^{-\frac{1}{2}}\mathcal{P}_h v\|_{H}\leq c\|v\|_{\dot{H}^{-1}}$ for $v\in\dot{H}^{-1}$ (cf. \cite[(13)]{kruse2014consistency}). Together with the Sobolev embedding $L^2(D)\hookrightarrow \dot{H}^{-1}$, this yields
	\begin{equation*}
		\begin{aligned}
			J_1&\leq \Dt\sum_{j=1}^{n}\Big\|A_h^{\frac{1}{2}}S_{h,\Dt}^{n-j+1}A_h^{-\frac{1}{2}}\mathcal{P}_h\big(F(u(t_{j-1},\xi_j))-F(u^{j}_{h,\xi_j})\big)\Big\|_{L^2(\Omega,H)}\\
			&\leq c\Dt\sum_{j=1}^{n}(t_n-t_{j-1})^{-\frac{1}{2}}\big\|F(u(t_{j-1},\xi_j))-F(u^{j}_{h,\xi_j})\big\|_{L^2(\Omega,\dot{H}^{-1})}\\
			&\leq c\Dt\sum_{j=1}^{n}(t_n-t_{j-1})^{-\frac{1}{2}}\big\|F(u(t_{j-1},\xi_j))-F(u^{j}_{h,\xi_j})\big\|_{L^2(\Omega,H)}\\
			&\leq c\Dt\sum_{j=1}^{n}(t_n-t_{j-1})^{-\frac{1}{2}}\big\|u(t_{j-1},\xi_j)-u^{j}_{h,\xi_j}\big\|_{L^2(\Omega,H)}\\
			&\leq c\Dt\sum_{j=1}^{n}(t_n-t_{j-1})^{-\frac{1}{2}}\Big(\big\|S_{h,\xi_j\Dt}(u(t_{j-1})-u_h^{j-1})\big\|_{L^2(\Omega,H)}+\big\|\xi_j\Dt S_{h,\xi_j\Dt}(F(u(t_{j-1}))-F(u_h^{j-1}))\big\|_{L^2(\Omega,H)}\\
			&\quad+\big\|S_{h,\xi_j\Dt}(G(u(t_{j-1}))-G(u_h^{j-1}))\Delta W_{\xi_j}^{j-1}\big\|_{L^2(\Omega,H)}\Big)=:J_{1,1}+J_{1,2}+J_{1,3}.
		\end{aligned}
	\end{equation*}
	We proceed to estimate $J_{1,1}$, $J_{1,2}$, and $J_{1,3}$ separately. For the first term $J_{1,1}$, employing $\|S_{h,\xi_j\Dt}\|_{\L(H)}\leq c$ to derive
	\begin{equation*}
		J_{1,1}\leq c\Dt\sum_{j=1}^{n}(t_n-t_{j-1})^{-\frac{1}{2}}\big\|u(t_{j-1})-u_h^{j-1}\big\|_{L^2(\Omega,H)}.
	\end{equation*}
	To estimate $J_{1,2}$, by virtue of $\xi_j\sim\mathcal{U}(0,1)$, $\Dt\in(0,T)$, $\|S_{h,\xi_j\Dt}\|_{\L(H)}\leq c$, and \eqref{Fcondition}, we obtain
	\begin{equation*}
		\begin{aligned}
			J_{1,2}&\leq c\Dt\sum_{j=1}^{n}(t_n-t_{j-1})^{-\frac{1}{2}}\| S_{h,\xi_j\Dt}\|_{\L(H)}\big\|F(u(t_{j-1}))-F(u_h^{j-1})\|_{L^2(\Omega,H)}\\
			&\leq c\Dt\sum_{j=1}^{n}(t_n-t_{j-1})^{-\frac{1}{2}}\big\|u(t_{j-1})-u_h^{j-1}\|_{L^2(\Omega,H)}.
		\end{aligned}
	\end{equation*}
	For $J_{1,3}$, we invoke the Burkholder-Davis-Gundy inequality, along with $\xi_j\sim\mathcal{U}(0,1)$, $\Dt\in(0,T)$, $\|S_{h,\xi_j\Dt}\|_{\L(H)}\leq c$, and \eqref{lipofG}, to deduce
	\begin{equation*}
		\begin{aligned}
			J_{1,3}&\leq c\Dt\sum_{j=1}^{n}(t_n-t_{j-1})^{-\frac{1}{2}}\Big(\mathbb{E_\xi}\Big[\mathbb{E}_W\big[\big\|(G(u(t_{j-1}))-G(u_h^{j-1}))\Delta W^{j-1}_{\xi_j}\big\|^2_H\big]\Big]\Big)^{\frac{1}{2}}\\
			&\leq cT^{\frac{1}{2}}\Dt\sum_{j=1}^{n}(t_n-t_{j-1})^{-\frac{1}{2}}\Big(\mathbb{E_\xi}\Big[\big\|G(u(t_{j-1}))-G(u_h^{j-1})\big\|^2_{L^2(\Omega_{W},\L_0^2)}\Big]\Big)^{\frac{1}{2}}\\
			&\leq c\Dt\sum_{j=1}^{n}(t_n-t_{j-1})^{-\frac{1}{2}}\big\|u(t_{j-1})-u_h^{j-1}\|_{L^2(\Omega,H)}.
		\end{aligned}
	\end{equation*}
	Combining the estimate for $J_{1,1}$, $J_{1,2}$, $J_{1,3}$ gives
	\begin{equation*}
		J_1\leq c\Dt\sum_{j=1}^{n}(t_n-t_{j-1})^{-\frac{1}{2}}\big\|u(t_{j-1})-u_h^{j-1}\big\|_{L^2(\Omega,H)}.
	\end{equation*}
	Furthermore, by Cauchy-Schwarz inequality and inequality $\Dt\sum_{j=1}^{n}(t_n-t_{j-1})^{-\frac{1}{2}}\leq 2 T^{\frac{1}{2}}$ (cf. \cite[(32)]{kruse2014consistency}), we obtain
	\begin{equation}\label{J3square}
		\begin{aligned}
			J_1^2&\leq c\Dt^2\Big(\sum_{j=1}^{n}(t_n-t_{j-1})^{-\frac{1}{2}}\Big)\Big(\sum_{j=1}^{n}(t_n-t_{j-1})^{-\frac{1}{2}}\|u(t_{j-1})-u_h^{j-1}\|^2_{L^2(\Omega,H)}\Big)\\
			&\leq c\Dt\sum_{j=1}^{n}(t_n-t_{j-1})^{-\frac{1}{2}}\big\|u(t_{j-1})-u_h^{j-1}\|^2_{L^2(\Omega,H)}.
		\end{aligned}
	\end{equation}
	For the estimation of $J_2$. Applying $\|S_{h,\xi_j\Dt}\|_{\L(H)}\leq c$, Burkholder-Davis-Gundy inequality, and \eqref{lipofG} yields
	\begin{equation*}
		\begin{aligned}
			J_2&\leq c\sum_{j=1}^{n}\Big(\mathbb{E_\xi}\Big[\mathbb{E}_W\big[\big\|\big(G(u(t_{j-1}))-G(u_h^{j-1})\big)(W(t_j)-W(t_{j-1}))\big\|_H^2\big]\Big]\Big)^{\frac{1}{2}}\\
			&\leq c\Dt^{\frac{1}{2}}\sum_{j=1}^{n}\Big(\mathbb{E_\xi}\Big[\big\|G(u(t_{j-1}))-G(u_h^{j-1})\big\|^2_{L^2(\Omega_{W},\L_0^2)}\Big]\Big)^{\frac{1}{2}}\leq c\Dt^{\frac{1}{2}}\sum_{j=1}^{n}\big\|G(u(t_{j-1}))-G(u_h^{j-1})\big\|_{L^2(\Omega,\L_0^2)}\\&
			\leq c\sum_{j=1}^{n}\|u(t_{j-1})-u_h^{j-1}\|_{L^2(\Omega,H)}\Dt^{\frac{1}{2}}.
		\end{aligned}
	\end{equation*}
	Further using Cauchy-Schwarz inequality and inequality $\Dt\sum_{j=1}^{n}(t_n-t_{j-1})^{-\frac{1}{2}}\leq 2 T^{\frac{1}{2}}$ allows to deduce
	\begin{equation}\label{J4square}
		\begin{aligned}
			J_2^2&\leq c\Dt\sum_{j=1}^{n}\|u(t_{j-1})-u_h^{j-1}\|^2_{L^2(\Omega,H)}\leq cT^\frac{1}{2}\Dt\sum_{j=1}^{n}(t_n-t_{j-1})^{-\frac{1}{2}}\big\|u(t_{j-1})-u_h^{j-1}\|^2_{L^2(\Omega,H)}.
		\end{aligned}
	\end{equation}
	For the term $J_3$, we first use  $\|S_{h,\Dt}^{n-j+1}\|_{\L(H)}\leq c$ and Burkholder-Davis-Gundy inequality to deduce
	\begin{equation}\label{J5ini}
		\begin{aligned}
			J_3&\leq c\sum_{j=1}^{n}\Big(\mathbb{E_\xi}\Big[\mathbb{E}_W\Big[\Big\|\int_{t_{j-1}}^{t_j}\Big(\int_{t_{j-1}}^{\tau}\big(G'(u(t_{j-1}))G(u(t_{j-1}))-G'(u_h^{j-1})G(u_h^{j-1})\big)dW(s)\Big)dW(\tau)\Big\|^2_H\Big]\Big]\Big)^{\frac{1}{2}}\\
			&\leq c\sum_{j=1}^{n}\Big(\mathbb{E_\xi}\Big[\int_{t_{j-1}}^{t_j}\Big\|\int_{t_{j-1}}^{\tau}\big(G'(u(t_{j-1}))G(u(t_{j-1}))-G'(u_h^{j-1})G(u_h^{j-1})\big)dW(s)\Big\|^2_{L^2(\Omega_{W},\L_0^2)}d\tau\Big]\Big)^{\frac{1}{2}}.
		\end{aligned}
	\end{equation}
	Now we are guided to estimate the term  $\big\|\int_{t_{j-1}}^{\tau}\big(G'(u(t_{j-1}))G(u(t_{j-1}))-G'(u_h^{j-1})G(u_h^{j-1})\big)dW(s)\big\|^2_{L^2(\Omega_{W},\L_0^2)}$. Applying Karhunen-Lo\`eve expansion \eqref{K-Lexpansion} to $W(s)$, and using the norm definition \eqref{HilbertHimitnorm}, the assumption \eqref{deriveg2} and It{\^o} isometry, we obtain for $\tau\in[t_{j-1},t_j]$,
	\begin{equation}\label{J5mid}
		\begin{aligned}
			\Big\|&\int_{t_{j-1}}^{\tau}\big(G'(u(t_{j-1}))G(u(t_{j-1}))-G'(u_h^{j-1})G(u_h^{j-1})\big)dW(s)\Big\|^2_{L^2(\Omega_{W},\L_0^2)}\\&=\big\|\big(G'(u(t_{j-1}))G(u(t_{j-1}))-G'(u_h^{j-1})G(u_h^{j-1})\big)(W(\tau)-W(t_{j-1}))\big\|^2_{{L^2(\Omega_{W},\L_0^2)}}\\&
			=\mathbb{E}_W\Big[\sum_{i=1}^{\infty}\big\|(G'(u(t_{j-1}))G(u(t_{j-1}))-G'(u_h^{j-1})G(u_h^{j-1}))q_i^{\frac{1}{2}}\Big(\sum_{k=1}^{\infty}q_k^{\frac{1}{2}}\phi_k\big(\beta_k(\tau)-\beta_k(t_{j-1})\Big)\phi_i\big\|^2_{H}\Big]
			\\
			&
			=(\tau-t_{j-1})\mathbb{E}_W\Big[\sum_{i=1}^{\infty}q_i\Big\|\big(G'(u(t_{j-1}))G(u(t_{j-1}))-G'(u_h^{j-1})G(u_h^{j-1})\big)\Big(\sum_{k=1}^{\infty}q_k^{\frac{1}{2}}\phi_k\Big)\phi_i\Big\|^2_{H}\Big]\\
			&\leq(\tau-t_{j-1})\mathbb{E}_W\Big[\sum_{i=1}^{\infty}\sum_{k=1}^{\infty}q_iq_k\Big\|\big(G'(u(t_{j-1}))G(u(t_{j-1}))-G'(u_h^{j-1})G(u_h^{j-1})\big)\phi_i\phi_k\Big\|^2_{H}\Big]\\
			&\leq(\tau-t_{j-1})\mathbb{E}_{W}\Big[\Big\|G'(u(t_{j-1}))G(u(t_{j-1}))-G'(u_h^{j-1})G(u_h^{j-1})\Big\|^2_{HS(Q^{\frac{1}{2}}(H),\L_0^2)}\Big]\leq c\Dt\big\|u(t_{j-1})-u_h^{j-1}\big\|^2_{{L^2(\Omega_W,H)}}.
		\end{aligned}
	\end{equation}
	Substituting \eqref{J5mid} into \eqref{J5ini} yields
	\begin{equation}\label{J5square}
		J_3^2\leq c\Dt^2\sum_{j=1}^{n}\big\|u(t_{j-1})-u_h^{j-1}\big\|^2_{L^2(\Omega,H)}\leq cT^{\frac{3}{2}}\Dt\sum_{j=1}^{n}(t_n-t_{j-1})^{-\frac{1}{2}}\|u(t_{j-1})-u_h^{j-1}\|^2_{L^2(\Omega,H)}.
	\end{equation}
	Finally, by combining the bounds derived in \eqref{condition32}, \eqref{J3square}, \eqref{J4square}, and \eqref{J5square}, we obtain the desired estimate in \eqref{Nonlinearstability}. This concludes the proof.
\end{proof}

Before deriving the strong convergence rate of the proposed scheme, we also need to establish an a priori bound for the difference between $\mathscr{G}(\cdot)$ and $G(\cdot)$. This estimate, stated in Lemmas~\ref{priorestimate} and~\ref{lemmalast}, plays a crucial role in the subsequent error analysis.
\begin{lemma}\label{priorestimate}
	Let $u_0\in L^p(\Omega_{W},\dot{H}^{2})$, $p\ge2$, be an $\mathcal{F}^W_0$-measurable random variable. Then there exists a constant $c>0$, independent of $h$ and $\Dt$, and an arbitrarily small number $\varepsilon_0>0$, such that
	\be\label{priorinequality}
	\|G(u(\tau))-\mathscr{G}(\tau, u(t_{j-1}))\|_{{L^2(\Omega,\L_0^2)}}\leq c(\tau-t_{j-1})^{1-\varepsilon_0},\ \ \forall \  t_{j-1}\leq \tau\leq t_j.
	\ee
\end{lemma}
\begin{proof}
	We begin by applying a Taylor expansion to the nonlinear diffusion term $G(\cdot)$, yielding
	\be\label{Taylor1}
	G(u(\tau))=G(u(t_{j-1}))+G'(u(t_{j-1}))(u(\tau)-u(t_{j-1}))+\theta_1(t_{j-1},\tau),
	\ee
	where the remainder term $\theta_1(t_{j-1},\tau)$ is given by
	\bex
	\theta_1(t_{j-1},\tau):=\int_{0}^{1}(1-r)G''\big(u(t_{j-1})+r(u(\tau)-u(t_{j-1}))\big)(u(\tau)-u(t_{j-1}))^2dr.
	\eex
	Define the auxiliary term
	\bex
	\theta_2(t_{j-1},\tau):=u(\tau)-u(t_{j-1})-G(u(t_{j-1}))(W(\tau)-W(t_{j-1})).
	\eex
	Substituting the identity \bex
	u(\tau)-u(t_{j-1})=\theta_2(t_{j-1},\tau)+G(u(t_{j-1}))(W(\tau)-W(t_{j-1}))
	\eex
	into \eqref{Taylor1}, and recalling that \bex
	\mathscr{G}(\tau, u(t_{j-1}))=G(u(t_{j-1}))+G'(u(t_{j-1}))G(u(t_{j-1}))(W(\tau)-W(t_{j-1})),
	\eex
	we arrive at
	\begin{equation*}
		\begin{aligned}
			G(u(\tau))&=G(u(t_{j-1}))+G'(u(t_{j-1}))G(u(t_{j-1}))(W(\tau)-W(t_{j-1}))+G'(u(t_{j-1}))\theta_2(t_{j-1},\tau)+\theta_1(t_{j-1},\tau)\\
			&=\mathscr{G}(\tau,u(t_{j-1}))+G'(u(t_{j-1}))\theta_2(t_{j-1},\tau)+\theta_1(t_{j-1},\tau).
		\end{aligned}
	\end{equation*}
	Note that $G'(u(t_{j-1}))$ is independent of $\theta_2(t_{j-1},\tau)$. By applying the bound \eqref{deriveg1}, we obtain
	\begin{equation}\label{thetaestimate:G}
		\begin{aligned}
			&\big\|G(u(\tau))-\mathscr{G}(\tau;u(t_{j-1}))\big\|_{{L^2(\Omega,\L_0^2)}}\leq c\big(\big\|G'(u(t_{j-1}))\theta_2(t_{j-1},\tau)\big\|_{{L^2(\Omega,\L_0^2)}}+\big\|\theta_1(t_{j-1},\tau)\big\|_{L^2(\Omega,\L_0^2)}\big)\\
			&\leq c\Big(\big(\mathbb{E}_\xi\big[\mathbb{E}_W[\big\|G'(u(t_{j-1}))\big\|^2_{{\L(H,\L_0^2)}}\big\|\theta_2(t_{j-1},\tau)\big\|^2_{{H}}\big]\big]\big)^{\frac{1}{2}}+\big\|\theta_1(t_{j-1},\tau)\big\|_{L^2(\Omega,\L_0^2)}\Big)\\
			&\leq c\Big(\big(\mathbb{E_\xi}\big[\big\|\theta_2(t_{j-1},\tau)\big\|^2_{{L^2(\Omega_W,H)}}\big]\big)^{\frac{1}{2}}+\big(\mathbb{E_\xi}\big[\big\|\theta_1(t_{j-1},\tau)\big\|^2_{L^2(\Omega_W,\L_0^2)}\big]\big)^{\frac{1}{2}}\Big)
		\end{aligned}
	\end{equation}
	We first estimate $\big\|\theta_2(t_{j-1},\tau)\big\|_{{L^2(\Omega_W,H)}}$. Define
	\bex
	\theta_3(t_{j-1},\tau):=u(\tau)-u(t_{j-1})-\int_{t_{j-1}}^{\tau}G(u(r))dW(r).
	\eex 
	According to \cite[Lemma 10.27, (10.39)]{LORD2014}, there exists an infinitesimal number $\varepsilon_0>0$ such that
	\begin{equation}\label{needresult}
		\|\theta_3(t_{j-1},\tau)\|_{L^2(\Omega_{W},H)}=\Big\| u(\tau)-u(t_{j-1})-\int_{t_{j-1}}^{\tau}G(u(\tau))dW(\tau)\Big\|_{L^{2}(\Omega_W,H)}\leq c(\tau-t_{j-1})^{1-\varepsilon_0}.
	\end{equation}
	Observe that $\theta_3(t_{j-1},\tau)$ can be rewritten as
	\begin{equation*}
		\begin{aligned}
			\theta_3(t_{j-1},\tau) =& u(\tau)-u(t_{j-1})-G(u(t_{j-1}))(W(\tau)-W(t_{j-1}))
			- \int_{t_{j-1}}^{\tau}\big(G(u(r))-G(u(t_{j-1}))\big)dW(r)\\
			=& \theta_2(t_{j-1},\tau)- \int_{t_{j-1}}^{\tau}\big(G(u(r))-G(u(t_{j-1}))\big)dW(r).
		\end{aligned}
	\end{equation*}
	A use of the  triangle inequality gives
	\bex
	\big\|\theta_2(t_{j-1},\tau)\big\|^2_{{H}} 
	\le c\Big(
	\big\|\theta_3(t_{j-1},\tau)\big\|^2_{{H}} 
	+
	\Big\|\int_{t_{j-1}}^{\tau}\big(G(u(r))-G(u(t_{j-1}))\big)dW(r)\Big\|^2_{{H}}\Big).
	\eex
	Taking expectations $\mathbb{E}_W[\cdot]$ on both sides and applying the It\^o isometry, along with \eqref{needresult}, \eqref{lipofG}, and the temporal regularity estimate \eqref{holderregularity}, yields
	\begin{equation}
		\begin{aligned}\label{theta2est}
			\big\|\theta_2(t_{j-1},&\tau)\big\|^2_{{L^2(\Omega_W,H)}}\leq c\Big( \big\|\theta_3(t_{j-1},\tau)\big\|^2_{{L^2(\Omega_W,H)}}+\mathbb{E}_W\Big[\Big\|\int_{t_{j-1}}^{\tau}\Big(G(u(r))-G(u(t_{j-1}))\big)dW(r)\Big\|^2_{{H}}\Big]\Big)\\&
			\leq  c\big\|\theta_3(t_{j-1},\tau)\big\|^2_{{L^2(\Omega_W,H)}}+c\int_{t_{j-1}}^{\tau}\mathbb{E}_W\big[\big\|G(u(r))-G(u(t_{j-1}))\big\|^2_{{\L_0^2}}\big]dr\\
			&\leq c(\tau-t_{j-1})^{2(1-\varepsilon_0)}+c(\tau-t_{j-1})\max_{t_{j-1}\leq r\leq\tau}\big\|u(r)-u(t_{j-1})\big\|^2_{{L^2(\Omega_W,H)}}\\&
			\leq c(\tau-t_{j-1})^{2(1-\varepsilon_0)}+c(\tau-t_{j-1})\max_{t_{j-1}\leq r\leq\tau}\big\|u(r)-u(t_{j-1})\big\|^2_{{L^2(\Omega_W,\dot{H}^1)}}\leq c (\tau-t_{j-1})^{2(1-\varepsilon_0)}.
		\end{aligned}
	\end{equation}
	It remains to estimate $\big\|\theta_1(t_{j-1},\tau)\big\|_{{L^2(\Omega_W,H)}}$.
	By the $\L^{\otimes2}(H,\L_0^2)$-norm definition and invoking \eqref{deriveg1}, we have
	\begin{equation*}
		\begin{aligned}
			\|\theta_1(t_{j-1},\tau)\|_{{\L_0^2}}&\leq \int_{0}^{1}(1-r)\big\|G''\big(u(t_{j-1})+r(u(\tau)-u(t_{j-1}))\big)(u(\tau)-u(t_{j-1}))^2\big\|_{{\L_0^2}}dr\\
			&\leq \int_{0}^{1}(1-r)\big\|G''\big(u(t_{j-1})+r(u(\tau)-u(t_{j-1}))\big)\big\|_{{\L^{\otimes2}(H,\L_0^2)}}\big\|u(\tau)-u(t_{j-1})\big\|^2_{{H}}dr\\
			&\leq c\int_{0}^{1}(1-r)\big\|u(\tau)-u(t_{j-1})\big\|^2_{{H}}dr\leq c\big\|u(\tau)-u(t_{j-1})\big\|^2_{{H}}.
		\end{aligned}
	\end{equation*}
	Applying the temporal regularity estimate \eqref{holderregularity} yields
	\be\label{theta3est}
	\big\|\theta_1(t_{j-1},\tau)\big\|_{{L^2(\Omega_W,\L_0^2)}}\leq c\Big(\mathbb{E}_W\big[\big\|u(\tau)-u(t_{j-1})\big\|^4_{H}\big]\Big)^{\frac{1}{2}}\leq c(\tau-t_{j-1}).
	\ee
	Substituting the bounds \eqref{theta2est} and \eqref{theta3est} into \eqref{thetaestimate:G} completes the proof.
\end{proof}
\begin{lemma}\label{lemmalast}
	Let $u_0\in L^p(\Omega_{W},\dot{H}^{2})$, $p\ge2$, be an $\mathcal{F}^W_0$-measurable random variable. Then there exists a constant $c>0$, independent of $h$ and $\Dt$, and an arbitrarily small number $\varepsilon_0>0$ such that
	\begin{equation*}
		\max_{n\in\{1,\dots,N\}}\Big\|\sum_{j=1}^{n}S_{h,\Dt}^{n-j+1}\int_{t_{j-1}}^{t_j}\big(G(u(\tau))-\mathscr{G}(\tau,u(t_{j-1}))\big)dW(\tau)\Big\|_{L^2(\Omega,H)}\leq c\Dt^{1-\varepsilon_0}.
	\end{equation*}
\end{lemma}
\begin{proof}
	Making use of Burkholder-Davis-Gundy inequality, together with the operator bound \eqref{operatoruniformbound} and the prior estimate \eqref{priorinequality}, we obtain
	\begin{equation*}
		\begin{aligned}
			&\Big\|\sum_{j=1}^{n}S_{h,\Dt}^{n-j+1}\int_{t_{j-1}}^{t_j}\big(G(u(\tau))-\mathscr{G}(\tau,u(t_{j-1}))\big)dW(\tau)\Big\|^2_{L^2(\Omega,H)}\\&
			\leq c\sum_{j=1}^{n}\Big(\mathbb{E_\xi}\Big[\int_{t_{j-1}}^{t_j}\big\|G(u(\tau))-\mathscr{G}(\tau,u(t_{j-1}))\big\|^2_{L^2(\Omega_{W},\L_0^2)}d\tau\Big]\Big)\\
			&\leq c\sum_{j=1}^{n}\int_{t_{j-1}}^{t_j}(\tau-t_{j-1})^{2(1-\varepsilon_0)}d\tau\leq c\Dt^{2(1-\varepsilon_0)}, \ \ n\in\{1,\dots,N\},
		\end{aligned}
	\end{equation*}
	where $\varepsilon_0>0$ is an arbitrarily small number. This completes the proof.
\end{proof}
Now we are in a position to derive the fully discrete error bound, as stated in the following theorem.
\begin{theorem}\label{needprooftheorem}
	Let $u(t)$ and $u_h^n$ be the mild solution and numerical solution given respectively by \eqref{for17} and \eqref{randomizedMilstein}. Suppose that $u_0\in L^p(\Omega_{W},\dot{H}^{2})$, $p\ge2$, is an $\mathcal{F}^W_0$-measurable random variable. Then there exists a constant $c>0$, independent of $\Dt\in(0,T)$ and $h\in(0,1)$, such that the following strong error estimate holds:
	\begin{equation*}
		\max_{n\in\{0,\dots,N\}}\big\| u(t_n)-u^n_{h}\big\|_{{L^2(\Omega,H)}}\leq c\big(\Delta t^{1-\varepsilon_0}+h^{2-\varepsilon_0}\big),
	\end{equation*}
	where $\varepsilon_0$ is an arbitrarily small constant.
\end{theorem}
\begin{proof}
	From scheme \eqref{randomizedMilstein}, with the initial condition $u_h^0=P_hu_0$, we obtain
	\begin{equation*}
		u_h^n
		= S_{h,\Delta t}\Big(u_h^{n-1}
		+ \Delta t F(u_{h,\xi_n}^n)
		+ \int_{t_{n-1}}^{t_n} \mathscr{G}(\tau,u_h^{n-1})\, dW(\tau)\Big),
	\end{equation*}
	where the intermediate value $u_{h,\xi_n}^n$ is given by
	\begin{equation*}
		u_{h,\xi_n}^n
		= S_{h,\xi_n\Delta t}\Big(u_h^{n-1}
		+ \xi_n\Delta t F(u_h^{n-1})
		+ G(u_h^{n-1})\Delta W^{n-1}_{\xi_{n}}\Big).
	\end{equation*}
According to the \eqref{Bhdtdiscrete}, the numerical scheme can be written compactly as
	\begin{equation*}
		u_h^n = S_{h,\Delta t}u_h^{n-1}
		+ B_{h,\Delta t}(u_h^{n-1},\xi_n).
	\end{equation*}
	We define the error at time level $n$ by
	\begin{equation*}
		e_h^n := u(t_n)-u_h^n,\ \ n=1,\dots,N.
	\end{equation*}
	Subtracting the numerical solution $u_h^n$ from the exact solution $u(t_n)$ yields the error recursion
	\begin{equation}\label{fullerror}
		e_h^n
		= S_{h,\Delta t} e_h^{n-1}
		+ \Big(B_{h,\Delta t}(u(t_{n-1}),\xi_n)
		- B_{h,\Delta t}(u_h^{n-1},\xi_n)\Big)
		+ R_{h,\Delta t}(u(t_n)),
	\end{equation}
	where $B_{h,\Delta t}(u(t_{n-1}),\xi_n)$, $B_{h,\Delta t}(u_h^{n-1},\xi_n)$, and $u(t_{n-1},\xi_n)$ are defined in \eqref{Bhdtcontinue}, \eqref{Bhdtdiscrete}, and \eqref{PSI}, respectively, and the remainder term $R_{h,\Delta t}(u(t_n))$ is given by
	\begin{equation}\label{reminder}
		R_{h, \Delta t} (u(t_{n}))
		 := u (t_{n}) - S_{h, \Delta t} u (t_{n-1}) - B_{h, \Delta t} ( u(t_{n-1}), \xi_{n}), 
	\end{equation}
	Recurring the equation \eqref{fullerror} gives
	\begin{equation}\label{errorrec}
		e_h^n = S_{h,\Delta t}^n e_h^0
		+ \sum_{j=1}^n S_{h,\Delta t}^{n-j}
		\Big(B_{h,\Delta t}(u(t_{j-1}),\xi_j)
		- B_{h,\Delta t}(u_h^{j-1},\xi_j)\Big)  + \sum_{j=1}^n S_{h,\Delta t}^{n-j}
		R_{h,\Delta t}(u(t_j)), \ \ n=1,\dots,N,
	\end{equation}
	where $e_h^0=u_0-\mathcal{P}_h u_0$. Taking the $L^2(\Omega,H)$-norm on both sides of the equation \eqref{errorrec} and using the inequality
	$(a+b+c)^2\le c(a^2+b^2+c^2)$, we arrive at
	\begin{align*}
		\big\|e_h^n\big\|_{L^2(\Omega,H)}^2
		&\le c\big\|S_{h,\Delta t}^n e_h^0\big\|_{L^2(\Omega,H)}^2  + c\Big\|\sum_{j=1}^n S_{h,\Delta t}^{n-j}
		\big(B_{h,\Delta t}(u(t_{j-1}),\xi_j)
		- B_{h,\Delta t}(u_h^{j-1},\xi_j)\big)\Big\|_{L^2(\Omega,H)}^2 \\
		&\quad + c\Big\|\sum_{j=1}^n S_{h,\Delta t}^{n-j}
		R_{h,\Delta t}(u(t_j))\Big\|_{L^2(\Omega,H)}^2.
	\end{align*}
	From the estimate \eqref{Nonlinearstability}, we have
	\begin{equation*}
		\Big\|\sum_{j=1}^n S_{h,\Delta t}^{n-j}
		\big(B_{h,\Delta t}(u(t_{j-1}),\xi_j)
		- B_{h,\Delta t}(u_h^{j-1},\xi_j)\big)\Big\|_{L^2(\Omega,H)}^2
		\le c\Delta t \sum_{j=1}^n (t_n-t_{j-1})^{-1/2}
		\big\|e_h^{j-1}\big\|_{L^2(\Omega,H)}^2.
	\end{equation*}
	Consequently,
	\begin{align*}
		\big\|e_h^n\big\|_{L^2(\Omega,H)}^2
		&\le c\big\|S_{h,\Delta t}^n e_h^0\big\|_{L^2(\Omega,H)}^2
		+ c 
		\Big\|\sum_{j=1}^n S_{h,\Delta t}^{n-j}
		R_{h,\Delta t}(u(t_j))\Big\|_{L^2(\Omega,H)}^2 + c\Delta t \sum_{j=1}^n
		(t_n-t_{j-1})^{-1/2}\big\|e_h^{j-1}\big\|_{L^2(\Omega,H)}^2.
	\end{align*}
	An application of Gronwall's inequality \cite[Lemma 3.9]{kruse2014consistency} yields
	\begin{equation*}
		\big\|e_h^n\big\|_{L^2(\Omega,H)}^2
		\le c\big\|S_{h,\Delta t}^n e_h^0\big\|_{L^2(\Omega,H)}^2
		+ c 
		\Big\|\sum_{j=1}^n S_{h,\Delta t}^{n-j}
		R_{h,\Delta t}(u(t_j))\Big\|_{L^2(\Omega,H)}^2.
	\end{equation*}
	Denote $I_1:=\big\|S_{h,\Delta t}^n e_h^0\big\|_{L^2(\Omega,H)}$. By the mild formulation (see \eqref{for17}) of the exact solution $u(t_n)$, we obtain
	\begin{align*}
		& \Big\|\sum_{j=1}^n S_{h,\Delta t}^{n-j} R_{h,\Delta t}(u(t_j))\Big\|_{L^2(\Omega,H)} \leq \big\|(S(t_n)-S_{h,\Delta t}^n)u_0\big\|_{L^2(\Omega,H)} \\
		&\quad+ \Big\|\sum_{j=1}^n \int_{t_{j-1}}^{t_j}
		(S(t_n-\tau)-S_{h,\Delta t}^{n-j+1})F(u(\tau))\,d\tau\Big\|_{L^2(\Omega,H)} \\
		&\quad + \Big\|\sum_{j=1}^n \int_{t_{j-1}}^{t_j}
		(S(t_n-\tau)-S_{h,\Delta t}^{n-j+1})G(u(\tau))\,dW(\tau)\Big\|_{L^2(\Omega,H)} \\
		&\quad + \Big\|\sum_{j=1}^n S_{h,\Delta t}^{n-j}
		\Big(\int_{t_{j-1}}^{t_j} S_{h,\Delta t}F(u(\tau))\,d\tau
		+ \int_{t_{j-1}}^{t_j} S_{h,\Delta t}G(u(\tau))\,dW(\tau)
		- B_{h,\Delta t}^j(u(t_{j-1}),\xi_j)\Big)\Big\|_{L^2(\Omega,H)}\\&
		\quad=:I_2+I_3+I_4+I_5.
	\end{align*}
	Therefore
	\begin{equation}\label{fullerrordecomp}
		\big\|e_h^n\big\|_{L^2(\Omega,H)}
		\le I_1 + I_2 + I_3 + I_4 + I_5.
	\end{equation}
By \eqref{operatoruniformbound} and \eqref{bound1} (see Lemma \ref{lemma1result} in Appendix), we deduce 
\begin{equation}\label{errorI1}
I_1=\big\|S_{h,\Dt}^n(u_0-\mathcal{P}_hu_0)\big\|_{L^2(\Omega,H)}\leq c\|S_{h,\Dt}^n\|_{\L(H)}\|u_0-\mathcal{P}_h u_0\|_{L^2(\Omega_{W},H)}\leq c h^{2}\|u_0\|_{L^2(\Omega_{W},\dot{H}^{2})}.
\end{equation}
From \eqref{bound2}--\eqref{bound4} (see Lemma \ref{lemma1result} in Appendix), we have
\begin{equation}\label{errorI2-I4}
		I_2+I_3+I_4
		\le c(h^{2-\varepsilon_0}+\Delta t^{1-\varepsilon_0}),
\end{equation}
where $\varepsilon_0>0$ denotes an arbitrarily small number. It remains to estimate $I_5$. According to \eqref{Bhdtcontinue} and \eqref{PSI}, $I_5$ can be bounded by
	\begin{equation*}
		I_5 \le I_{5,1}+I_{5,2}+I_{5,3},
	\end{equation*}
	where 
		\begin{align*}
			 I_{5,1}:=&\Big\|\sum_{j=1}^{n}S_{h,\Dt}^{n-j+1}\int_{t_{j-1}}^{t_j}\big(F(u(\tau))-F(u(t_{j-1}+\xi_j\Dt))\big)d\tau\Big\|_{L^2(\Omega,H)},\\
			I_{5,2}:=&\Big\|\Dt\sum_{j=1}^{n}S_{h,\Dt}^{n-j+1}\big(F(u(t_{j-1}+\xi_j\Dt))-F(u(t_{j-1},\xi_j))\big)\Big\|_{L^2(\Omega,H)},\\
			I_{5,3}:=&\Big\|\sum_{j=1}^{n}S_{h,\Dt}^{n-j+1}\int_{t_{j-1}}^{t_j}\big(G(u(\tau))-\mathscr{G}(\tau,u(t_{j-1}))\big)dW(\tau)\Big\|_{L^2(\Omega,H)}.
		\end{align*}
Using \eqref{martigalederivative} (see Lemma \ref{lastlemma1} in Appendix) yields
	\begin{equation*}
		I_{5,1} \le c\Delta t.
	\end{equation*}
By Lemma \ref{lastlemma2}, there exists an infinitesimal positive number $\varepsilon_0>0$ such that
	\begin{equation*}
		I_{5,2} \le c(h^{2-\varepsilon_0}+\Delta t^{1-\varepsilon_0}).
	\end{equation*}
From Lemma \ref{lemmalast}, there exists an arbitrarily small number $\varepsilon_0>0$ such that
	\begin{equation*}
		I_{5,3} \le c\Delta t^{1-\varepsilon_0}.
	\end{equation*}
Therefore, combining all above estimates of $I_{5,1}-I_{5,3}$, we arrive at 
\be\label{I5estimate}
I_5\leq c (h^{2-\varepsilon_0}+\Delta t^{1-\varepsilon_0}).
\ee
By \eqref{fullerrordecomp}, \eqref{errorI1}, \eqref{errorI2-I4}, and \eqref{I5estimate}, we arrive at
	\begin{equation*}
	\|e_h^n\|_{L^2(\Omega,H)}
		\le C(\Delta t^{1-\varepsilon_0}+h^{2-\varepsilon_0}),\ n=1,\dots,N.
	\end{equation*}
	This completes the proof.
\end{proof}
\begin{remark}
	In \cite[(4)]{kruse2014consistency}, Kruse introduced a \emph{semi-implicit Milstein-Galerkin finite element} scheme for the full discretization of semilinear SPDEs, and established strong convergence rates of order  $\mathcal{O}(h^{1+r}+\Dt^{\frac{1+r}{2}})$ with $r\in[0,1)$ (see \cite[Theorem 1.1]{kruse2014consistency}). In the present work, Theorem~\ref{needprooftheorem} shows that the proposed numerical scheme \eqref{randomizedMilstein} achieves strong convergence of order  (1-$\varepsilon_0$) in time and order (2-$\varepsilon_0$) in space, which correspond to the rates obtained in \cite[Theorem 1.1]{kruse2014consistency} when
	 $r=1-\varepsilon_0$. 
	 
	 A key distinction between the two approaches lies in the analytical assumptions employed. The convergence analysis in \cite{kruse2014consistency} fundamentally relies on the Fr{\'e}chet differentiability of the nonlinear drift term $F$ (see  \cite[Assumption 2.3]{kruse2014consistency}). In contrast, the analysis of scheme \eqref{randomizedMilstein} developed in this work avoids such differentiability requirements by leveraging a martingale-type inequality (see Lemma~\ref{lastlemma1}). 
\end{remark}

\section{Numerical experiments}\label{numericalresult}
Since the main contribution of this work is the proof of spatio-temporal strong convergence rates for the numerical scheme \eqref{randomizedMilstein}, this section presents numerical experiments to corroborate the theoretical convergence orders. 

We first consider a one-dimensional problem on $D=(0,1)$:
		\begin{equation}
			\begin{aligned}\label{demension1}
				\mathrm{d}u(t)&=\Big(-Au+\frac{1}{1+|u|}\Big)\mathrm{d}t+\delta u\mathrm{d}W(t),\    t\in(0,T],\\
				u(0)&=u_0,
			\end{aligned}
		\end{equation}
where $Au:=-u_{xx}$ is equipped with homogeneous Dirichlet boundary conditions, and $\delta>0$ denotes the noise intensity. The nonlinear drift is $F(u)=\frac{1}{1+|u|}$ and the diffusion coefficient is $G(u)=u$. It is straightforward to verify that $F$ and $G$ satisfy \eqref{Fcondition} and Assumption \ref{Gassumption}. The driving $Q$-Wiener process $W(t)$ is defined by
		\bex
		W(t):=\sum\limits_{j=1}^\infty \sqrt{q_j}\phi_j \beta_j(t)=\sum\limits_{j=1}^\infty \sqrt{2q_j} \sin(j\pi x)\beta_j(t), \ x\in\bar{D},
		\eex
		where $q_j:=\mathcal{O}(j^{-2})$ and $\{\beta_j(t)\}_{j\ge 1}$ are independent standard Brownian motions. The initial datum is chosen as $u_0(x)=\sin(2\pi x)$, $x\in \bar{D}$. In all experiments we set $T=0.1$ and $\delta=0.5$.
		
		We investigate the strong (mean-square) error at the final time $T$. Since an analytic solution of problem \eqref{demension1} is unavailable, we employ a reference solution computed on a sufficiently fine space-time grid. For the temporal-convergence test, the reference solution is computed using $h=1/128$ and $\Delta t=10^{-6}$; for the spatial convergence test, we use $h=1/512$ and $\Dt=10^{-5}$. The mean-square error,  $\|u(T)-u^N_h\|_{L^2(\Omega,H)}$, is approximated by the empirical average over 500 independent samples:
		\bex
		\big\|u(T)-u^N_h\big\|_{L^2(\Omega,H)}\approx\Big(\frac{1}{500}\sum_{j=1}^{500}\big\|u^{\rm{ref}}_j(T)-u^N_{j,h}(T)\big\|_{H}^2\Big)^\frac{1}{2}=:u_{\rm{error}},
		\eex
		where $u^{\rm{ref}}_j(T)$ and $u_{j,h}^N(T)$ denote the $j$-th realizations of the reference and fully discrete numerical solutions, respectively. 
		
		Since $G(u)=u$ satisfies the standard commutativity-type condition (see, e.g., \cite[Assumption 3, Remark 1]{jentzen2015milstein}), the L{\'e}vy area terms that typically arise in Milstein-type discretizations can be avoided. Moreover, to approximate the It{\^o} stochastic integrals involving the term $\mathscr{G}(\cdot)$, we follow the standard discretization procedure described in \cite[pp.~464]{LORD2014}. We compute \eqref{demension1} using the proposed scheme \eqref{randomizedMilstein} and, for comparison, the semi-implicit Milstein-Galerkin finite element scheme in \cite[(4)]{kruse2014consistency}. The numerical results are reported in Table \ref{1comparisionoftsc}. The experimental orders of convergence (EOC) are computed by
		\bex
		\mbox{EOC}_{\rm{time}}=\log_2\Big(\frac{ u_{\rm{error}}(\Dt_i)}{u_{\rm{error}}(\Dt_{i+1})}\Big), \ \  \mbox{EOC}_{\rm{space}}=\log_2\Big(\frac{u_{\rm{error}}(h_i)}{u_{\rm{error}}(h_{i+1})}\Big), \ \ i=1,2,3,4,
		\eex
		where $u_{\rm{error}}(\Dt_i)$ and $u_{\rm{error}}(h_i)$ denote the numerical errors corresponding to the time step $\Dt_i$ and spatial mesh size $h_i$, respectively. From Table \ref{1comparisionoftsc}, we observe that the proposed scheme \eqref{randomizedMilstein} exhibits temporal and spatial convergence orders approaching one and two, respectively. These rates are consistent with those of the semi-implicit Milstein--Galerkin finite element scheme and are in good agreement with the theoretical analysis.

Furthermore, we consider a two-dimensional numerical example. We revisit problem \eqref{demension1} on the spatial domain $D=(0,1)^2$. In this setting, we take  $A=-(\frac{\partial^2}{\partial x_1^2}+\frac{\partial^2}{\partial x_2^2}),\ x_1,x_2\in(0,1)$, and let $W(t)$ be $Q$-Wiener process of the form
\bex
W(t):=\sum_{j_1,j_2=1}^{\infty}\sqrt{q_{j_1,j_2}}\sin(j_1\pi x_1)\sin(j_2\pi x_2)\beta_{j_1,j_2}(t),
\eex
with $q_{j_1,j_2}=\exp(-\frac{j_1^2+j_2^2}{200})$ \footnote[2]{This exponentially decaying choice of $q_{j_1,j_2}$ is commonly adopted to model smooth trace-class noise in two spatial dimensions; see, for example, \cite[Example 10.12]{LORD2014}.}, and $\beta_{j_1,j_2}(t)$ denoting  independent standard Brownian motions. 

In the two dimensional case, the numerical solution computed on a sufficiently refined space-time grid is taken as the reference solution. Specifically, for the temporal convergence test, we use $\Delta t = 10^{-6}$ together with a uniform $64\times64$ spatial mesh. For the spatial convergence test, we take $\Delta t = 10^{-4}$ and a uniform $256\times256$ mesh. Similar to the one-dimensional case, we compute the numerical error $u_{\mathrm{error}}$ and the corresponding EOC. The results are reported in Table~\ref{2dimension}. From Table~\ref{2dimension}, we observe that the proposed scheme achieves temporal and spatial convergence rates approaching first order and second order, respectively. These numerical observations are in full agreement with the theoretical error estimates established in Theorem~\ref{needprooftheorem}.

\begin{table*}[htbp]
	\centering
	\caption{Time (upper table) and space (lower table) convergence rates in one-dimensional case.}
	\label{1comparisionoftsc}
	\begin{tabular}{ccccc}
		\toprule
		\multirow{2}{*}{Time step $\Dt$} & \multicolumn{2}{c}{drift-randomized scheme \eqref{randomizedMilstein}} & \multicolumn{2}{c}{semi-implicit scheme \cite[(4)]{kruse2014consistency}}\\
		\cmidrule(r){2-3} \cmidrule(r){4-5} 
		&  $u_{\rm{error}}$ &   $\mbox{EOC}_{\rm{time}}$
		&  $u_{\rm{error}}$ &   $\mbox{EOC}_{\rm{time}}$\\
		\midrule
		$\Dt_1$=	1.00E-2	&6.4816E-4  &  {\bf --} &  6.5529E-4 &  {\bf --}  \\
		$\Dt_2$=	5.00E-3	&3.2754E-4  &  0.9847  &  3.2682E-4  & 1.0036 \\
		$\Dt_3$=	2.50E-3	&1.5517E-4 & 1.0779  & 1.6278E-4  & 1.0056 \\
		$\Dt_4$=	1.25E-3	&8.6343E-5 & 0.8457  &  8.4236E-5 & 0.9504 \\
		$\Dt_5$=	6.25E-4	&4.5713E-5 & 0.9175  &  4.3098E-5  &  0.9668\\
		\bottomrule		
	\end{tabular}\quad\quad
	\begin{tabular}{ccccc}
		\toprule
		\multirow{2}{*}{Mesh size $h$} & \multicolumn{2}{c}{drift-randomized scheme \eqref{randomizedMilstein}} 
		& \multicolumn{2}{c}{semi-implicit scheme \cite[(4)]{kruse2014consistency}}\\
		\cmidrule(r){2-3} \cmidrule(r){4-5} 
		&  $u_{\rm{error}}$ &   $\mbox{EOC}_{\rm{space}}$ 
		&  $u_{\rm{error}}$ &   $\mbox{EOC}_{\rm{space}}$\\
		\midrule
		$h_1\ $=1/16	&3.4857E-3  &  {\bf --}  &  3.7389E-3 &  {\bf --}  \\
		$h_2\ $=1/32	&8.7693E-4  & 1.9909  &  9.3813E-4  & 1.9948 \\
		$h_3\ $=1/64	&2.1875E-4 & 2.0032  & 2.3346E-4  & 2.0066 \\
		$h_4$=1/128	&5.5412E-5 & 1.9810  &  5.6735E-5 & 2.0409 \\
		$h_5$=1/256	&1.3926E-5 & 1.9924  & 1.3789E-5  & 2.0407 \\
		\bottomrule		
	\end{tabular}
\end{table*}
\begin{table*}[htbp]
	\centering
	\caption{Time (left table) and space (right table) convergence rates in two-dimensional case.}
	\label{2dimension}
	\begin{tabular}{ccc}
		\toprule
		\multirow{2}{*}{Time step $\Dt$} & \multicolumn{2}{c}{drift-randomized scheme \eqref{randomizedMilstein}}\\
		\cmidrule(r){2-3} 
		&  $u_{\rm{error}}$ &   $\mbox{EOC}_{\rm{time}}$\\
		\midrule
		$\Dt_1$=	2.5000E-4	&1.2844E-2  &  {\bf --}\\
		$\Dt_2$=	1.2500E-4	&6.6730E-3  &  0.9447\\
		$\Dt_3$= 6.2500E-5	&3.4879E-3 & 0.9360\\
		$\Dt_4$= 3.1250E-5	&1.8867E-3 & 0.8865\\
		$\Dt_5$= 1.5625E-5	&9.8509E-4 & 0.9375\\
		\bottomrule		
	\end{tabular}\quad\quad
	\begin{tabular}{ccc}
		\toprule
		\multirow{2}{*}{Mesh size $h$} & \multicolumn{2}{c}{drift-randomized scheme \eqref{randomizedMilstein}}\\ 
		\cmidrule(r){2-3} 
		&  $u_{\rm{error}}$ &   $\mbox{EOC}_{\rm{space}}$\\ 
		\midrule
		$h_1$ =1/8	&1.0621E-2  &  {\bf --}\\
		$h_2$ =1/16	&2.7014E-3  &  1.9751\\ 
		$h_3$ =1/32	&7.0233E-4 &  1.9435\\
		$h_4$ =1/64	&1.8165E-4 & 1.9510\\
		$h_5$ =1/128	&4.7122E-5 & 1.9467\\
		\bottomrule		
	\end{tabular}
\end{table*}

\section{Conclusion}\label{sec13}
In this paper, we carried out a rigorous strong convergence analysis for the \emph{drift-randomized Milstein--Galerkin finite element} fully discrete scheme. Without imposing any differentiability assumptions on the nonlinear drift term, we established strong spatio-temporal convergence rates for the proposed method that are comparable to those of the classical semi-implicit Milstein--Galerkin finite element scheme. Numerical experiments were conducted to corroborate the theoretical results.

\section{Appendix}\label{appendix}
In appendix, we present some crucial lemmas that is used throughout the paper.
\begin{lemma}[Well-posedness of the scheme \eqref{randomizedMilstein}]\label{wellposedness}
	\label{lem:wellposedness}
	Assume that the drift term $F$ satisfies the global Lipschitz and linear growth conditions \eqref{Fcondition}, and that the diffusion coefficient $G$ satisfies Assumption \ref{Gassumption}. Let $u_0 \in L^p(\Omega_W,H)$, $p\ge2$, and set $u_h^0 := P_h u_0 \in L^p(\Omega,V_h)$.  Then, for each $n=1,\dots,N$, the recursion \eqref{randomizedMilstein} admits a unique solution
	\bex
	u_{h,\xi_n}^n\in L^p(\Omega,V_h)\ \ \mbox{and} \ \,u_h^n \in L^p(\Omega,V_h).
	\eex
	Moreover, the numerical solution $u_h^n$ is $\mathcal F^{\pi_{\Delta t}}_n$-measurable.
\end{lemma}

\begin{proof}
	We proceed by induction on $n$.
	
	\medskip
	\noindent
	\emph{Step 1: existence and uniqueness at a single time step.}
	Let $u_h^{n-1} \in L^p(\Omega,V_h)$ be given. Since the discrete operator $A_h:V_h\to V_h$ is symmetric and positive definite, the operator
	\[
	I+\alpha \Delta t A_h : V_h \to V_h
	\]
	is invertible for any $\alpha\in[0,1]$. Consequently, the resolvent operators
	\[
	S_{h,\alpha\Delta t} := (I+\alpha \Delta t A_h)^{-1}P_h
	\]
	are well-defined bounded linear mappings on $H$.
	
	The intermediate value $u_{h,\xi_n}^n$ is given explicitly by
	\[
	u_{h,\xi_n}^n
	=
	S_{h,\xi_n\Delta t}
	\Big(
	u_h^{n-1}
	+
	\xi_n\Delta t\,F(u_h^{n-1})
	+
	G(u_h^{n-1})\,\Delta W^{n-1}_{\xi_n}
	\Big),
	\]
	and hence is uniquely defined.  
	Similarly, $u_h^n$ is uniquely determined by
	\[
	u_h^n
	=
	S_{h,\Delta t}
	\Big(
	u_h^{n-1}
	+
	\Delta t\,F(u_{h,\xi_n}^n)
	+
	\int_{t_{n-1}}^{t_n} \mathscr{G}(\tau,u_h^{n-1})\,dW(\tau)
	\Big).
	\]
	Thus, the scheme \eqref{randomizedMilstein} is well-defined at each time step.
	
	\medskip
	\noindent
	\emph{Step 2: measurability and adaptedness.}
	Assume that $u_h^{n-1}$ is $\mathcal F^{\pi_{\Delta t}}_{n-1}$-measurable. Since
	\[
	\Delta W^{n-1}_{\xi_n}
	=
	W(t_{n-1}+\xi_n\Delta t)-W(t_{n-1})
	\]
	is measurable with respect to $\mathcal F^W_{t_{n-1}+\xi_n\Delta t}\vee\sigma(\xi_n)$ and $u_h^{n-1}$ is independent of $\xi_n$, it follows that $u_{h,\xi_n}^n$ is measurable with respect to
	$\mathcal F^W_{t_{n-1}+\xi_n\Delta t}\vee\sigma(\xi_n)$.
	
	Moreover, the stochastic integral
	\[
	\int_{t_{n-1}}^{t_n} \mathscr{G}(\tau,u_h^{n-1})\,dW(\tau)
	\]
	is well-defined since $\mathscr{G}(\cdot,u_h^{n-1})$ is predictable and square-integrable by Assumption \ref{Gassumption}. Hence this term is $\mathcal F^W_{t_n}$-measurable.
	
	Since $S_{h,\Delta t}$ is deterministic and linear, we conclude that $u_h^n$ is measurable with respect to
	\[
	\mathcal F^{\pi_{\Delta t}}_n
	=
	\mathcal F^W_{t_n}\vee \sigma(\xi_1,\dots,\xi_n).
	\]
	
	\medskip
	\noindent
	\emph{Step 3: $L^p$-integrability.}
	The $L^p(\Omega,V_h)$-integrability of $u_{h,\xi_n}^n$ and $u_h^n$ follows from the linear growth of $F$ and $G$, the boundedness of the resolvent operators \eqref{operatoruniformbound}, and standard Burkholder--Davis--Gundy inequalities.
	
	This completes the proof.
\end{proof}

\begin{lemma}\label{lemma1result}
	Let $u_0\in L^p(\Omega_{W},\dot{H}^{2})$, $p\ge2$, be an $\mathcal{F}^W_0$-measurable random variable. Then there exists a constant $c>0$, independent of $h$ and $\Dt$, and an arbitrarily small parameter $\varepsilon_0>0$ such that
	\begin{align}
		&\|u_0-\mathcal{P}_h u_0\|_{L^2(\Omega_{W},H)}\leq c h^{2}\|u_0\|_{L^2(\Omega_{W},\dot{H}^{2})}, \label{bound1}\\
		&\max_{n\in\{1,\dots,N\}}\big\|\big(S(t_n)-S_{h,\Dt}^n\big)u_0\big\|_{L^2(\Omega_{W},H)}\leq c(h^{2}+\Dt)\|u_0\|_{L^2(\Omega_{W},\dot{H}^2)},\label{bound2}\\
		&\max_{n\in\{1,\dots,N\}}\Big\|\sum_{j=1}^{n}\int_{t_{j-1}}^{t_j}\big(S(t_n-\tau)-S_{h,\Dt}^{n-j+1}\big)F(u(\tau))d\tau\Big\|_{L^2(\Omega_{W},H)}\leq c(h^{2-\varepsilon_0}+\Dt^{1-\varepsilon_0}),\label{bound3}\\
		&\max_{n\in\{1,\dots,N\}}\Big\|\sum_{j=1}^{n}\int_{t_{j-1}}^{t_j}\big(S(t_n-\tau)-S_{h,\Dt}^{n-j+1}\big)G(u(\tau))dW(\tau)\Big\|_{L^2(\Omega_{W},H)}\leq c(h^{2-\varepsilon_0}+\Dt^{1-\varepsilon_0}).\label{bound4}
	\end{align}
\end{lemma}
\begin{proof}
	The estimates \eqref{bound1}, \eqref{bound2}, and \eqref{bound4} are established in \cite[Lemma 5.1]{kruse2019randomized}, \cite[Lemma 5.3]{kruse2019randomized}, and \cite[Lemma 5.5]{kruse2014consistency}, respectively. It remains to prove the estimate in \eqref{bound3}. To this end, we apply the triangle inequality to obtain
	\begin{equation}\label{45estimate}
		\begin{aligned}
			\Big\|\sum_{j=1}^{n}\int_{t_{j-1}}^{t_j}\big(S(t_n-&\tau)-S_{h,\Dt}^{n-j+1}\big)F(u(\tau))d\tau\Big\|_{L^2(\Omega_{W},H)}\\
			&\leq\Big\|\sum_{j=1}^{n}\int_{t_{j-1}}^{t_j}\big(S(t_n-\tau)-S_{h,\Dt}^{n-j+1}\big)\big(F(u(\tau))-F(u(t_{j-1}))\big)d\tau\Big\|_{L^2(\Omega_{W},H)}\\&\quad+\Big\|\sum_{j=1}^{n}\int_{t_{j-1}}^{t_j}\big(S(t_n-\tau)-S_{h,\Dt}^{n-j+1}\big)F(u(t_{j-1}))d\tau\Big\|_{L^2(\Omega_{W},H)}=:J_1+J_2.
		\end{aligned}
	\end{equation}
	We estimate $J_1$ and $J_2$ separately. For  $J_1$, using the smoothing property \eqref{d-operatoresti-positive} with $\rho=2-\varepsilon_0$ and $\mu=0$, the condition \eqref{Fcondition}, and the temporal H{\"o}lder regularity of the mild solution \eqref{holderregularity}, we obtain
	\begin{equation*}
		\begin{aligned}
			J_1&\leq c\big(h^{2-\varepsilon_0}+\Dt^{1-\varepsilon_0}\big)\sum_{j=1}^{n}\int_{t_{j-1}}^{t_j}(t_n-\tau)^{-(1-\varepsilon_0)}\|u(\tau)-u(t_{j-1})\|_{L^2(\Omega_{W},H)}d\tau\\
			&\leq cT^{\frac{1}{2}}(h^{2-\varepsilon_0}+\Dt^{1-\varepsilon_0})\int_{0}^{t_n}(t_n-\tau)^{-(1-\varepsilon_0)}d\tau\leq c(h^{2-\varepsilon_0}+\Dt^{1-\varepsilon_0}),
		\end{aligned}
	\end{equation*}
	where $\varepsilon_0>0$ is arbitrarily small. 
	
	For the second term $J_2$, we again apply \eqref{d-operatoresti-positive} with $\rho=2-\varepsilon_0$ and $\mu=0$, together with \eqref{Fcondition} and the regularity bound \eqref{uregularity}, to deduce
	\begin{equation*}
		\begin{aligned}
			J_2&\leq c\big(h^{2-\varepsilon_0}+\Dt^{1-\varepsilon_0}\big)\sum_{j=1}^{n}\int_{t_{j-1}}^{t_j}(t_n-\tau)^{-(1-\varepsilon_0)}\big(1+\|u(t_{j-1})\|_{L^2(\Omega_{W},H)}\big)d\tau\\
			&\leq c(h^{2-\varepsilon_0}+\Dt^{1-\varepsilon_0})\int_{0}^{t_n}(t_n-\tau)^{-(1-\varepsilon_0)}d\tau\leq c(h^{2-\varepsilon_0}+\Dt^{1-\varepsilon_0}).
		\end{aligned}
	\end{equation*}
	Combining the above estimates for $J_1$ and $J_2$, and recalling \eqref{45estimate}, we conclude that the bound in \eqref{bound3} holds.
\end{proof}
\begin{lemma}\label{lastlemma1}
	Let $u_0\in L^p(\Omega_{W},\dot{H}^{2})$, $p\ge2$, be an $\mathcal{F}^W_0$-measurable random variable. Then there exists a constant $c>0$, independent of $h$ and $\Dt$, such that
	\begin{equation}\label{martigalederivative}
		\max_{n\in\{1,\dots,N\}}\Big\|\sum_{j=1}^{n}S_{h,\Dt}^{n-j+1}\int_{t_{j-1}}^{t_j}\big(F(u(\tau))-F(t_{j-1}+\xi_j\Dt)\big)d\tau\Big\|_{L^2(\Omega,H)}\leq c\Dt.
	\end{equation}
\end{lemma}
\begin{proof}
	By following the same lines of argument as in \cite[Lemma 5.7]{kruse2019randomized}, we can derive the estimate \eqref{martigalederivative}. The key ingredient of the proof is the use of a discrete-time martingale inequality \cite[Proposition 2.2]{kruse2019randomized}, which allows to avoid relying on a Taylor expansion of the nonlinear drift term $F(\cdot)$. It completes the proof.
\end{proof}
\begin{lemma}\label{lastlemma2}
	Let $u_0\in L^p(\Omega_{W},\dot{H}^{2})$, $p\ge2$, be an $\mathcal{F}^W_0$-measurable random variable. Then there exists a constant $c>0$, independent of $h$ and $\Dt$, and an arbitrarily small number $\varepsilon_0>0$ such that
	\begin{equation*}
		\max_{n\in\{1,\dots,N\}}\Big\|\Dt\sum_{j=1}^{n}S_{h,\Dt}^{n-j+1}\big(F(u(t_{j-1}+\xi_j\Dt))-F(u(t_{j-1},\xi_j))\big)\Big\|_{L^2(\Omega,H)}\leq c(h^{2-\varepsilon_0}+\Dt^{1-\varepsilon_0}).
	\end{equation*}
\end{lemma}
\begin{proof}
	Applying \eqref{operatoruniformbound} and the Lipschitz continuity of $F(\cdot)$ from \eqref{Fcondition}, we obtain
	\begin{equation}\label{formula24}
		\begin{aligned}
			\max_{n\in\{1,\dots,N\}}&\Big\|\Dt\sum_{j=1}^{n}S_{h,\Dt}^{n-j+1}\big(F(u(t_{j-1}+\xi_j\Dt))-F(u(t_{j-1},\xi_j))\big)\Big\|_{L^2(\Omega,H)}\\&
			\leq c\Dt\sum_{j=1}^{N}\big\|u(t_{j-1}+\xi_j\Dt)-u(t_{j-1},\xi_j)\big\|_{L^2(\Omega,H)}.
		\end{aligned}
	\end{equation}
	To estimate the $\big\|u(t_{j-1}+\xi_j\Dt)-u(t_{j-1},\xi_j)\big\|_{L^2(\Omega,H)}$, we insert  the variation-of-constants formula
	\begin{equation*}
		\begin{aligned}
			u(t_{j-1}+\xi_j\Dt)=&u(t_{j-1})S(\xi_j\Dt)+\int_{t_{j-1}}^{t_{j-1}+\xi_j\Dt}S(t_{j-1}+\xi_j\Dt-\tau)F(u(\tau))d\tau\\&+\int_{t_{j-1}}^{t_{j-1}+\xi_j\Dt}S(t_{j-1}+\xi_j\Dt-\tau)G(u(\tau))dW(\tau)
		\end{aligned}
	\end{equation*}
	and \eqref{PSI}, i.e.,
	\bex
	u(t_{j-1},\xi_j)= S_{h,\xi_j\Dt} u(t_{j-1})+\xi_j \Dt S_{h,\xi_j\Dt} F(u(t_{j-1}))+S_{h,\xi_j\Dt} G(u(t_{j-1}))\Delta W^{j-1}_{\xi_j}
	\eex
	into  \eqref{formula24} to derive
	\begin{equation}\label{4.15}
		\begin{aligned}
			\big\|&u(t_{j-1}+\xi_j\Dt)-u(t_{j-1},\xi_j)\big\|_{L^2(\Omega,H)}\leq \big\|\big(S(\xi_j\Dt)-S_{h,\xi_j\Dt}\big)u(t_{j-1})\big\|_{L^2(\Omega,H)}\\&+\Big\|\int_{t_{j-1}}^{t_{j-1}+\xi_j\Dt}\big(S(t_{j-1}+\xi_j\Dt-\tau)F(u(\tau))-S_{h,\xi_j\Dt}F(u(t_{j-1}))\big)d\tau\Big\|_{L^2(\Omega,H)}\\
			&+\Big\|\int_{t_{j-1}}^{t_{j-1}+\xi_j\Dt}\big(S(t_{j-1}+\xi_j\Dt-\tau)G(u(\tau))-S_{h,\xi_j\Dt}G(u(t_{j-1}))\big)dW(\tau)\Big\|_{L^2(\Omega,H)}=:J_1+J_2+J_3.
		\end{aligned}
	\end{equation}
	We now estimate each term $J_1$, $J_2$, $J_3$ separately. For the first term $J_1$, employing \eqref{uregularity}, and \eqref{d-operatoresti-positive} with $\rho=\mu=2-\varepsilon_0$, gives
	\begin{equation}\label{J1estimate}
		\begin{aligned}
			J_1&=\Big(\mathbb{E_\xi}\Big[\big\|\big(S(\xi_j\Dt)-S_{h,\xi_j\Dt}\big)u(t_{j-1})\big\|^2_{L^2(\Omega_{W},H)}\Big]\Big)^{\frac{1}{2}}=\Big(\frac{1}{\Dt}\int_{0}^{\Dt}\big\|\big(S(\theta)-S_{h,\theta}\big)u(t_{j-1})\big\|^2_{L^2(\Omega_{W},H)}d\theta\Big)^{\frac{1}{2}}\\
			&\leq c(h^{2-\varepsilon_0}+\Dt^{1-\varepsilon_0})\Big(\frac{1}{\Dt}\int_{0}^{\Dt}\big\|A^{1-\varepsilon_0}u(t_{j-1})\|^2_{L^2(\Omega_{W},H)}d\theta\Big)^{\frac{1}{2}}\leq c(h^{2-\varepsilon_0}+\Dt^{1-\varepsilon_0}),
		\end{aligned}
	\end{equation}
	where $\varepsilon_0>0$ is an arbitrarily small number.
	
	The term $J_2$ can be bounded by using $\|S(\cdot)\|_{\L(H)}\leq c$, the linear growth condition of $F(\cdot)$ from \eqref{Fcondition}, the regularity result \eqref{uregularity}, and $\xi_j\sim\mathcal{U}(0,1)$:
	\begin{equation}\label{J2estimate}
		\begin{aligned}
			J_2=&\Big\|\int_{t_{j-1}}^{t_{j-1}+\xi_j\Dt}\big(S(t_{j-1}+\xi_j\Dt-\tau)F(u(\tau))-S_{h,\xi_j\Dt}F(u(t_{j-1}))\big)d\tau\Big\|_{L^2(\Omega,H)}\\
			&\leq \Big(\mathbb{E_\xi}\Big[\Big(\int_{t_{j-1}}^{t_{j-1}+\xi_j\Dt}\big(\|F(u(\tau))\|_{L^2(\Omega_{W},H)}+\|F(u(t_{j-1}))\|_{L^2(\Omega_W,H)}\big)d\tau\Big)^2\Big]\Big)^{\frac{1}{2}} \\&
			\leq c\big(1+\sup_{0\leq t\leq T}\|u(t)\|_{L^2(\Omega_{W},H)}\big)\Dt\leq c\Dt.
		\end{aligned}
	\end{equation}
	To estimate $J_3$, we first add and subtract a suitable term, yielding 
	\begin{equation*}
		\begin{aligned}
			J_3=&\Big\|\int_{t_{j-1}}^{t_{j-1}+\xi_j\Dt}\big(S(t_{j-1}+\xi_j\Dt-\tau)G(u(\tau))-S_{h,\xi_j\Dt}G(u(t_{j-1}))\big)dW(\tau)\Big\|_{L^2(\Omega,H)}\\
			&\leq \Big\|\int_{t_{j-1}}^{t_{j-1}+\xi_j\Dt}S(t_{j-1}+\xi_j\Dt-\tau)\big(G(u(\tau))-G(u(t_{j-1}))\big)dW(\tau)\Big\|_{L^2(\Omega,H)}\\
			&\quad+\Big\|\int_{t_{j-1}}^{t_{j-1}+\xi_j\Dt}\big(S(t_{j-1}+\xi_j\Dt-\tau)-S_{h,\xi_j\Dt}\big)G(u(t_{j-1}))dW(\tau)\Big\|_{L^2(\Omega,H)}=:J_{3,1}+J_{3,2}.
		\end{aligned}
	\end{equation*}
	We now estimate $J_{3,1}$ and $J_{3,2}$ individually. For $J_{3,1}$, applying Burkholder-Davis-Gundy inequality, $\|S(\cdot)\|_{\L(H)}\leq c$, \eqref{lipofG},  \eqref{holderregularity}, and $\xi_j\sim\mathcal{U}(0,1)$, yields
	\begin{equation*}
		\begin{aligned}
			J_{3,1}&\leq c\Big(\mathbb{E_\xi}\Big[\int_{t_{j-1}}^{t_{j-1}+\xi_j\Dt}\big\|G(u(\tau))-G(u(t_{j-1}))\big\|^2_{L^2(\Omega_W,\L_0^2)}d\tau\Big]\Big)^{\frac{1}{2}}\\
			&\leq c\Big(\mathbb{E_\xi}\Big[\int_{t_{j-1}}^{t_{j-1}+\xi_j\Dt}\big\|u(\tau)-u(t_{j-1})\big\|^2_{L^2(\Omega_W,H)}d\tau\Big]\Big)^{\frac{1}{2}}\leq c\Dt.
		\end{aligned}
	\end{equation*}
	For $J_{3,2}$, we again use Burkholder-Davis-Gundy inequality, together with the operator estimate \eqref{d-operatoresti-positive} for $\rho=2-\varepsilon_0$ and $\mu=1$, the growth bound \eqref{linearofG}, the regularity result \eqref{uregularity}, and $\xi_j\sim\mathcal{U}(0,1)$, to obtain
	\begin{equation*}
		\begin{aligned}
			J_{3,2}&\leq  c\Big(\mathbb{E_\xi}\Big[\int_{t_{j-1}}^{t_{j-1}+\xi_j\Dt}\big\|\big(S(t_{j-1}+\xi_j\Dt-\tau)-S_{h,\xi_j\Dt}\big)G(u(t_{j-1}))\big\|^2_{L^2(\Omega_W,\L_0^2)}d\tau\Big]\Big)^{\frac{1}{2}}\\&
			\leq c\big(h^{2-\varepsilon_0}+\Dt^{1-\varepsilon_0}\big)\Big(\mathbb{E_\xi}\Big[\int_{t_{j-1}}^{t_{j-1}+\xi_j\Dt}\big(t_{j-1}+\xi_j\Dt-\tau\big)^{\varepsilon_0-1}\big\|A^{\frac{1}{2}} G(u(t_{j-1}))\big\|^2_{L^2(\Omega_{W},\L_0^2)}d\tau\Big]\Big)^{\frac{1}{2}}
			\\&\leq c\big(h^{2-\varepsilon_0}+\Dt^{1-\varepsilon_0}\big)\big(1+\sup_{0\leq t\leq T}\|u(t)\|_{L^2(\Omega_{W},\dot{H}^1)}\big)\Big(\mathbb{E_\xi}\Big[\int_{t_{j-1}}^{t_{j-1}+\xi_j\Dt}\big(t_{j-1}+\xi_j\Dt-\tau\big)^{\varepsilon_0-1}d\tau\Big]\Big)^{\frac{1}{2}}\\
			&\leq c\big(h^{2-\varepsilon_0}+\Dt^{1-\varepsilon_0}\big),
		\end{aligned}
	\end{equation*}
	where $\varepsilon_0$ is an infinitesimal positive number.
	
	Combining the bounds for $J_{3,1}$ and $J_{3,2}$, we arrive at
	\begin{equation}\label{J3estimate}
		J_3\leq c\big(h^{2-\varepsilon_0}+\Dt^{1-\varepsilon_0}\big).
	\end{equation}
	Collecting the estimates \eqref{formula24}, \eqref{4.15}, \eqref{J1estimate}, \eqref{J2estimate}, and \eqref{J3estimate} completes the proof.
\end{proof}

\section*{}
\medskip
\noindent
{\bf CRediT authorship contribution statement}

\vspace{2mm}

{{\bf Xiao Qi:} Writing-original draft, Data curation, Formal analysis, Investigation, Methodology, Validation. {\bf Yue Wu:} Conceptualization, Formal analysis, Investigation, Methodology, Resources. 
{\bf Yubin Yan:} Conceptualization, Data curation, Formal analysis, Investigation, Methodology, Project administration, Resources, Supervision, Writing-review \& editing.}

\medskip
\noindent
{\bf Declaration of competing interest} 

\vspace{4mm}

The authors declare the following financial interests/personal relationships which may be considered as potential
competing interests: Xiao Qi is partially supported by the China Scholarship Council and the Research Fund of Jianghan University under Grant No. 2024JCYJ04.  Other authors have no known competing financial interests or personal relationships that could have appeared to influence the work reported in this paper.

\medskip
\noindent
{\bf Data availability} 

\vspace{4mm}

No data was used for the research described in the article.

\bibliography{sn-bibliography}

\end{document}